\documentclass[12pt,a4paper]{amsart}
\usepackage{amsmath, latexsym, enumerate, 
wasysym, 
bbm,
bm}
\usepackage{pict2e}
\usepackage{color}
\usepackage{amsthm}
\usepackage{amssymb}
\usepackage{amsbsy}
\usepackage{amsfonts}
\usepackage{amstext}
\usepackage{amscd}
\usepackage{graphicx}
\usepackage[dvips]{epsfig}
\usepackage{mathtools}
\DeclareMathOperator{\runs}{\mathrm{runs}}
\DeclareMathOperator{\cycleruns}{\mathrm{cruns}}
\DeclareMathOperator{\cycles}{\mathrm{c}}
\newcommand{\eulerbr}[2]{ \genfrac{\langle}{\rangle}{0pt}{}{#1}{#2}}

\usepackage{url}

\usepackage[a4paper,margin=2cm]{geometry}

\usepackage{tikz}
\usetikzlibrary{positioning}

\newlength{\standardunit}
\usepackage{stmaryrd}

\numberwithin{equation}{section}
\theoremstyle{plain}
\newtheorem{thm}{Theorem}[section]
\newtheorem{prop}[thm]{Proposition}
\newtheorem{cor}[thm]{Corollary}
\newtheorem{lem}[thm]{Lemma}
\theoremstyle{definition}
\newtheorem{exa}[thm]{Example}

\newtheorem{rem}[thm]{Remark}
\newtheorem{defi}[thm]{Definition}

\newcommand{\IC}{\mathbf{C}}
\newcommand{\IZ}{\mathbf{Z}}

\newcommand{\real}{\mathbb{R}}
\newcommand{\comp}{\mathbb{C}}
\newcommand{\SP}{\mathcal{P}}
\newcommand{\irr}[1]{#1_{\mathrm{irr}}}
\newcommand{\conn}[1]{#1_{\mathrm{conn}}}

\newcommand{\SG}{\mathfrak{S}}
\newcommand{\Cycles}{\mathcal{C}}

\newcommand{\NCP}{\mathrm{NC}}
\newcommand{\NC}{\NCP}

\newcommand{\MP}{\mathcal{M}}

\newcommand{\IP}{\mathcal{I}}
\newcommand{\I}{\IP}

\newcommand{\nestforest}{\tau}

\newcommand{\abs}[1]{\left\lvert #1 \right\rvert}  

\newcommand{\cseries}{\comp \llbracket z_1,\ldots,z_r \rrbracket}

\title{Relations between cumulants in noncommutative probability}
\author{Octavio Arizmendi}
\address{Centro de Investigaci\'on en Matem\'aticas\\ Jalisco S/N, Col. Valenciana., 36240 Guanajuato, Mexico}
\email{octavius@cimat.mx}
\author{Takahiro Hasebe}
\address{Department of Mathematics, Hokkaido University \\ Kita 10, Nishi 8, Kita-ku, Sapporo 060-0810, Japan}
\email{thasebe@math.sci.hokudai.ac.jp}
\author{Franz Lehner}
\address{Institute for  Mathematical  Structure Theory\\
Graz Technical University\\
Steyrergasse 30, 8010 Graz, Austria}
\email{lehner@math.tugraz.at}
\author{Carlos Vargas}
\address{Fachrichtung Mathematik\\
Universit\"at des Saarlandes\\ Postfach 151150, 66041 Saarbr\"ucken, Germany.
}
\email{carlos@math.uni-sb.de}
\thanks{T.~Hasebe partially supported by Austrian Science Foundation
  (FWF) project P25510-N26 and by Marie Curie Actions -- International Incoming Fellowships Project 328112 ICNCP.\\ 
 C.~Vargas supported by DFG-Deutsche Forschungsgemeinschaft, Project SP419/8-1.            
              }

\date{\today{}}
\begin{document}

\begin{abstract} 
  We express classical, free, Boolean and monotone cumulants in terms of each
  other, using combinatorics of heaps, pyramids, Tutte
  polynomials and permutations.
  We completely determine the coefficients of these formulas with the exception
  of the formula for classical cumulants in terms of monotone cumulants whose
  coefficients are only partially computed.  
\end{abstract}

\maketitle

\section{Introduction}
Cumulants provide a combinatorial description of independence of random
variables. While Fourier analysis is the tool of choice for most problems in
classical probability, cumulants are an indispensable ingredient for many
investigations in noncommutative probability.
An intriguing aspect of noncommutative probability is the existence of several
kinds of independence \cite{V85,SW97,M01,L04} with corresponding
cumulants introduced in \cite{V85,S94,SW97,HS11a,L04} sharing many common features. 
In a certain sense (which can be made precise,
see \cite{Speicher:1995:universal,Muraki:2002:universal})
these are the only ``natural'' notions of independence
and the combinatorics of cumulants in particular show very close analogies
between the different theories. Roughly speaking, one can pass from classical to
free/boolean/monotone independence by replacing the lattice of all set
partitions by noncrossing/interval/monotone partitions respectively.

On the other hand, the generating functions of cumulants correspond
to various transforms of probability measures and one major application is the calculation of
walk generating functions (or Green's functions) of certain graph products,
see \cite{Woess:2000:random} for details of the following concepts.
The cartesian product of graphs corresponds to classical convolution
as observed by Polya \cite{Polya:1921:Aufgabe},
the free product of graphs corresponds to Voiculescu's free convolution
\cite{Voiculescu:1986:addition,Woess:L1986:nearest,CartwrightSoardi:1986:random}
and the star product of graphs \cite[Section~9.7]{Woess:2000:random} 
corresponds to Boolean convolution 
\cite {Obata:2004:quantum}.
The comb product entered the graph theory literature rather recently
\cite{KrishnapurPeres:2004:recurrent} (a special case
having been considered earlier in physics, see \cite{WeissHavlin:1986:comb})
in order to construct an example of a recurrent random walk with so-called
\emph{finite collision property}.
It was observed
in \cite{AccardiBenGhorbalObata:2004:monotone} that this graph product corresponds to
monotone convolution.

The starting point of this paper is
the following relations between (univariate)
classical ($(\kappa_n)_{n\geq1}$), free ($(r_n)_{n\geq 1}$) and Boolean
($(b_n)_{n \geq 1}$) cumulants, shown by the third author \cite {L02} some time ago:
\begin{align}
b_n&=\sum_{\pi\in \irr\NC(n)}r_\pi, \label{eq:free2boolean}\\
r_n&=\sum_{\pi\in\conn\SP(n)}\kappa_\pi, \label{eq:class2free}\\
b_n&=\sum_{\pi\in\irr\SP(n)}\kappa_\pi.  \label{eq:class2Boolean}
\end{align}
where $\irr\NC(n)$, $\conn\SP(n)$, $\irr\SP(n)$ are, respectively, the
sets of irreducible noncrossing partitions, connected partitions and
irreducible partitions\footnote{Partitions and notations are defined in Section
  \ref{sec:partition}.}. Relation \eqref{eq:class2free} was used in
\cite{BelinschiBozejkoLehnerSpeicher:2011:normal} to attack the problem
of free infinite divisibility of the normal law.

We denote by $K_n, H_n, B_n, R_n$ the multivariate classical, monotone, Boolean
and free cumulants respectively. The univariate cumulants $\kappa_n, h_n,
b_n,r_n$ are obtained by evaluating the multivariate cumulants at $n$ copies
of a single variable.

Relation (\ref{eq:free2boolean}) was extended by Belinschi and Nica in \cite{BN08} to the case of multivariate cumulants $B_n,R_n$. In addition, they obtained the inverse formula:
\begin{equation}\label{eq:boolean2free:BeNi1}
R_n=\sum_{\pi\in \irr\NC(n)}(-1)^{\abs{\pi}-1}B_\pi.
\end{equation}
It is interesting to notice the similarity to formula (4) in \cite{LauveMastnak:2011:primitives}. 
We will give a different proof of \eqref{eq:boolean2free:BeNi1} in Section \ref{sec:proof1} which clarifies this coincidence. 

The extensions of (\ref{eq:class2free}) and (\ref{eq:class2Boolean}) to the
multivariate case can be shown by using the same proofs as in \cite {L02} for the
univariate case, see below for details. An interesting inverse formula for
(\ref{eq:class2Boolean}) was proved recently by M.~Josuat-Verg\`es 
\cite{JV13}, expressing classical cumulants in terms of free cumulants:
\begin{equation}\label{eq:free2class}
\kappa_n=\sum_{\pi\in\conn\SP(n)}(-1)^{\abs{\pi}-1}T_{G(\pi)}(1,0)r_{\pi}, 
\end{equation}
where $G(\pi)$ is the crossing graph of $\pi$ and $T_{G(\pi)}$ its Tutte
polynomial. The proof of (\ref{eq:free2class}) in \cite{JV13}
is also valid for the multivariate case.

The purpose of the present article is to complete the picture for the relations
between classical, Boolean, free and monotone cumulants, extending some
identities to the multivariate case.  More precisely, we are able to prove the
following cumulant identities.
\begin{thm}\label{mainthm1}
The following identities hold for multivariate cumulants:  
\begin{align}
B_n&=\sum_{\pi\in \irr\MP(n)} \frac{1}{\abs{\pi}!}H_\pi
= \sum_{\pi\in \irr\NC(n)} \frac{1}{\nestforest(\pi)!}H_\pi,
 \label{mb}  \\
R_n&=\sum_{\pi\in \irr\MP(n)}\frac{(-1)^{\abs{\pi}-1}}{\abs{\pi}!}H_\pi
= \sum_{\pi\in \irr\NC(n)} \frac{(-1)^{\abs{\pi}-1}}{\nestforest(\pi)!}H_\pi,
 \label{mf}
\end{align}
where $\irr\MP(n)$ is the set of irreducible monotone partitions. 
\end{thm}
\begin{thm}\label{mainthm2}
The following identities hold for univariate cumulants: 
\begin{align}
h_n&=\sum_{\pi\in \irr\NC(n)}\alpha_\pi r_\pi, \label{eq:fm}\\
h_n&=\sum_{\pi\in \irr\NC(n)} (-1)^{\abs{\pi}-1}\alpha_\pi b_\pi,\label{eq:bm}\\
h_n&=\sum_{\pi\in \irr\SP(n)}\alpha_{\bar{\pi}} \kappa_\pi, \label{eq:cm}
\end{align}
where $\bar{\sigma}\in\NC(n)$ denotes the noncrossing closure of $\sigma\in
\SP(n)$ and $\alpha_\pi$ is the linear part of the number of nonincreasing
labellings of the nesting forest of $\pi$ (which in the case of irreducible
partitions consists of precisely one tree). This quantity will be
defined rigorously in Section~\ref{sec:colored}.
\end{thm}
\begin{rem}
 Calculations indicate that a multivariate analogue of
 Theorem \ref{mainthm2} also holds, 
 but at present we do not know how to prove it. 
\end{rem}

The proof of the Boolean-to-classical cumulant formula follows the techniques
of the proof (\ref{eq:free2class}) used in  \cite{JV13}.
\begin{thm}\label{mainthm3} 
  \begin{equation}
    \label{eq:bool2class}
    K_n=\sum_{\pi\in\irr\SP(n)}(-1)^{\abs{\pi}-1}T_{\tilde{G}(\pi)}(1,0)B_{\pi},
  \end{equation}
where $\tilde{G}(\pi)$ is the anti-interval graph of $\pi$ and $T_{\tilde{G}(\pi)}$ is its Tutte polynomial (see Section \ref{sec:proof3}).
\end{thm}
Alternatively, the values of the Tutte polynomials
in \eqref{eq:free2class} and \eqref{eq:bool2class}
can be interpreted as certain pyramids in the sense of Cartier-Foata,
see Section~\ref{sec:proof3} for details.

Yet another interpretation  expresses classical cumulants in terms of Boolean
cumulants via  permutation statistics.
\begin{thm} 
  \label{thm:bool2classcycles}
  Denote by $\Cycles_n$ the set of cyclic permutations of order $n$ and let $\cycleruns(\sigma)$ be the set partition defined by the cycle runs of $\sigma$ (see Section \ref{permutation}). Then
  \begin{equation}\label{eq:bool2classcycles}
    K_n = \sum_{\sigma\in \Cycles_n} (-1)^{\#\cycleruns(\sigma)-1} B_{\cycleruns(\sigma)}. 
  \end{equation}
\end{thm}
There is a bijection between $\Cycles_n$ and $\{\sigma \in \SG_n \mid \sigma(1)=1\}$: given $\pi=(1,\pi(1),\dots,\pi^{n-1}(1)) \in\Cycles_n$, we define a permutation $\sigma(1)=1, \sigma(k)=\pi^{k-1}(1), ~2 \leq k \leq n$. Then we may rewrite Theorem~\ref{thm:bool2classcycles} into 
\begin{cor}\label{thm:bool2classpermutation1} Let $\SG_n$ be the set of permutations of order $n$, let $\runs(\sigma)$ be the set partition associated to the runs of $\sigma \in\SG_n$ and let $d(\sigma)$ be the number of descents of $\sigma \in\SG_n$ (see Section \ref{permutation}). Then 
  \begin{equation}\label{eq:bool2classpermutation0}
  K_n = \sum_{\substack{\sigma\in \SG_n\\ \sigma(1)=1}} (-1)^{d(\sigma)} B_{\runs(\sigma)}. 
  \end{equation}
\end{cor}

Understanding the coefficients of the remaining, monotone-to-classical cumulant formula $$K_n=\sum_{\pi\in\SP(n)}\beta(\pi) H_{\pi},$$
seems to require a more detailed treatment. We compute $\beta$ for some particular cases and list some of its properties. In particular, we show that the coefficients $\beta$ only depend on an anti-interval digraph, and moreover we show that
\begin{thm}\label{classical-monotone}
  \begin{enumerate}[\rm(1)]
   \item If $\pi$ is reducible, then $\beta(\pi)=0$.    
   \item 
    If $\pi$ is irreducible and has no nestings, 
    then $\beta(\pi)$ coincides with the
    coefficient
    $(-1)^{\abs{\pi}-1} T_{G(\pi)}(1,0)$ from formula
    \eqref{eq:free2class}.
   \item
      If $\pi\in \irr\NC$ and has depth 1 or 2, then
      $
      \beta(\pi)=
      \frac{(-1)^{\abs{\pi}-1}}{\abs{\pi}}. 
      $
  \end{enumerate}
\end{thm}

Thus various combinatorial objects contribute to the proofs and the
paper is organized accordingly. 
Types of partitions and notation are defined in Section
\ref{sec:partition}. 
We collect some combinatorial properties of monotone partitions in Section \ref{sec:monotonepartition}. 
Theorems \ref{mainthm1}, \ref{mainthm2}, \ref{mainthm3},
\ref{thm:bool2classcycles}, \ref{classical-monotone} are proved in Sections \ref{sec:proof1},
\ref{sec:colored}, \ref{sec:proof3}, \ref{permutation}, \ref{sec8}, respectively.

\section{Definitions and preliminary results}\label{sec:partition}
Concepts on partitions and ordered partitions are summarized below,
first let us recall some well known facts from the theory of \emph{posets} 
(partially ordered sets). 
For details on the latter the standard reference is \cite{Stanley:2012:enumerative1}.
\begin{prop}[Principle of M\"obius inversion]
  \label{prop:moebius}
  On any poset $(P,\leq)$ there is a unique \emph{M\"obius function}
  $\mu:P\times P \to \IZ$ such that for any pair of functions $f,g:P\to \IC$
  (in fact any abelian group in place of $\IC$) the identity
  \begin{equation}
    \label{eq:fx=sumgy}
  f(x)= \sum_{y\leq x} g(y)
  \end{equation}
  holds for every $x\in P$ if and only if
  \begin{equation}
  g(x)= \sum_{y\leq x} f(y)\,\mu(y,x). 
  \end{equation}
  In particular, if, given $f$, two functions $g_1$ and $g_2$ satisfy
  \eqref{eq:fx=sumgy}, then $g_1$ and $g_2$ coincide.
\end{prop}

\begin{defi}
  Let $P$ be a poset. A map $c:P\to P$ is called \emph{closure operator} if: 
  \begin{enumerate}
   \item it is \emph{increasing}, i.e., $x\leq c(x)$ for every $x\in P$; 
   \item it is \emph{order preserving}, i.e., if $x\leq y$ then $c(x)\leq c(y)$; 
   \item it is \emph{idempotent}, i.e., $c\circ c=c$.
  \end{enumerate}
\end{defi}
In the present paper
we will be concerned with posets (mostly lattices)
of set partitions  exclusively and make use of the noncrossing closure
and interval closure defined next.
\begin{defi}
  \label{def:partitions}
  \begin{enumerate}[\rm(1)]
  \item A \emph{partition} of a set is a decomposition into disjoint subsets,
   called \emph{blocks}.
   The set of partitions of the set $[n]:=\{1,\ldots,n\}$ is denoted by
   $\SP(n)$. 
   It is a lattice under refinement order with maximal element $\{[n]\}$
   denoted by $\hat{1}_n$ 
   and  minimal element $\{\{1\}, \dots, \{n\}\}$ is denoted by $\hat{0}_n$.  
   
   We write $\SP = \bigcup_{n \geq 1} \SP(n)$ and similar notations will be used such as $\NC$.  
  
  \item Any partition defines an equivalence relation
   on $[n]$ and vice versa. Given  $\pi\in \SP(n)$,
   $i \sim_\pi j$ holds if and only if there is a block $V \in \pi$ 
   such that $i,j \in V$.
  
   \item A partition $\pi\in\SP(n)$ is \emph{noncrossing} if there is no
    quadruple of elements $1\leq i<j<k<l\leq n$ such that $i\sim_\pi k$, $j\sim_\pi l$ and
    $i\not\sim_\pi j$.  The noncrossing partitions of order $n$ form a 
    sub-lattice
    which we denote by $\NC(n)$.
    
   \item For two blocks $V,W$ of a partition, we say $V$ is an
    \emph{inner block} of $W$ or equivalently $V$ nests inside $W$ or $W$ is an \emph{outer block} of
    $V$ if there are $i,j \in W$ such that $i<k<j$ for each $k \in V$. 
    
    \item The \emph{depth} of a block $V$ of a noncrossing partition is the number of blocks (including $V$ itself) which graphically cover the block $V$. The \emph{depth} of a noncrossing partition is the maximal depth among all the blocks. For example $\hat{1}_n$ has depth 1. 
  
   \item A block $V$ of a partition is called an \emph{interval block} if $V$ is of the form $V=\{k,k+1,\ldots, k+l\}$ for $k \geq 1$ and $0 \leq l \leq n-k$. We denote by $\mathit{IB}(n)$ the set of all interval blocks of
    $[n]$.
      
   \item An \emph{interval partition} is a partition $\pi$ for which every block is an interval.
    The set of interval partitions of $[n]$ is denoted by $\mathcal{I}(n)$
    and is a sub-lattice of $\SP(n)$. Sometimes these are called \emph{linear
      partitions} and in fact they are in obvious bijection with
    \emph{compositions} of a number $n$, i.e., sequences of integers $(k_1, k_2, \dots,
    k_r)$ such that $k_i>0$ and $k_1+k_2+\dots+k_r=n$.
   
   \item  The \emph{noncrossing closure} $\bar{\pi}$ of a partition $\pi$ 
    is the smallest noncrossing
    partition which dominates $\pi$.
   
   \item A partition $\pi$ is \emph{connected} if its noncrossing closure
    is equal to the maximal partition $\hat{1}_n$, or,
    equivalently, the diagram of $\pi$ is a
    connected graph.  
    The set of connected partitions is denoted by $\conn\SP(n)$.

   \item 
    The \emph{connected components} of a partition $\pi$ are the 
    connected sub-partitions of $\pi$, i.e., the partitions induced
    on the blocks of the noncrossing  closure $\bar\pi$.

   \item The \emph{interval closure} $\hat{\pi}$ of a partition $\pi$ 
    is the smallest interval
    partition which dominates $\pi$.

   \item A partition $\pi\in\SP(n)$ is \emph{irreducible} 
    if its interval closure
    is equal to the maximal partition $\hat{1}_n$. 
    For a noncrossing partition of $[n]$ this is equivalent to the
    property that $1\sim_\pi n$.   
    Every partition $\pi$ can be ``factored'' into irreducible factors which we
    denote by $\pi =\pi_1\cup \cdots \cup \pi_r$. 
    The factors $\pi_j$ are sub-partitions induced on the blocks
    of the interval closure $\hat\pi$.

    The sets of irreducible partitions and irreducible noncrossing partitions
    are respectively denoted by $\irr\SP(n)$ and $\irr\NC(n)$.  

  \end{enumerate}
\end{defi}
  Different types of partitions are shown in the following figure.

%
%
\begin{center}
  \begin{figure}[htbp]
    \label{fig:GeneralCumulants:Partitions}
    \begin{minipage}{.3\textwidth}
      \begin{center}
        \begin{picture}(80,20.4)(1,0)
          \put(10,0){\line(0,1){20.4}}
          \put(20,0){\line(0,1){8.4}}
          \put(30,0){\line(0,1){14.4}}
          \put(40,0){\line(0,1){8.4}}
          \put(50,0){\line(0,1){8.4}}
          \put(60,0){\line(0,1){14.4}}
          \put(70,0){\line(0,1){20.4}}
          \put(80,0){\line(0,1){8.4}}
          \put(20,8.4){\line(1,0){20}}
          \put(30,14.4){\line(1,0){30}}
          \put(10,20.4){\line(1,0){60}}
          \put(50,8.4){\line(1,0){30}}
        \end{picture}
        
        connected
      \end{center}
    \end{minipage}
%
%
    \begin{minipage}{.3\textwidth}
      \begin{center}
      \begin{picture}(90,20.4)(1,0)
        \put(10,0){\line(0,1){20.4}}
        \put(20,0){\line(0,1){8.4}}
        \put(30,0){\line(0,1){14.4}}
        \put(40,0){\line(0,1){8.4}}
        \put(50,0){\line(0,1){14.4}}
        \put(60,0){\line(0,1){8.4}}
        \put(70,0){\line(0,1){20.4}}
        \put(80,0){\line(0,1){8.4}}
        \put(90,0){\line(0,1){8.4}}
        \put(20,8.4){\line(1,0){20}}
        \put(30,14.4){\line(1,0){20}}
        \put(10,20.4){\line(1,0){60}}
        \put(60,8.4){\line(1,0){30}}
      \end{picture}

        irreducible
      \end{center}
    \end{minipage}
%
%
    \begin{minipage}{.3\textwidth}
      \begin{center}
        \begin{picture}(80,14.4)(1,0)
          \put(10,0){\line(0,1){14.4}}
          \put(20,0){\line(0,1){8.4}}
          \put(30,0){\line(0,1){8.4}}
          \put(40,0){\line(0,1){8.4}}
          \put(50,0){\line(0,1){14.4}}
          \put(60,0){\line(0,1){14.4}}
          \put(70,0){\line(0,1){8.4}}
          \put(80,0){\line(0,1){14.4}}
          \put(20,8.4){\line(1,0){20}}
          \put(10,14.4){\line(1,0){40}}
          \put(70,8.4){\line(1,0){0}}
          \put(60,14.4){\line(1,0){20}}
        \end{picture}
        
        noncrossing
      \end{center}
    \end{minipage}
    \caption{Typical partitions}
  \end{figure}
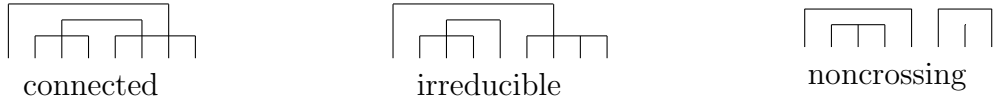
\end{center}


\begin{defi} 
 \begin{enumerate}[\rm(1)]
\item An \emph{ordered partition} is a pair $(\pi,\lambda)$ of a set partition $\pi$
and a linear order $\lambda$ on its blocks. An ordered partition can be regarded as a sequence of blocks: $(\pi,\lambda)=(V_1,\ldots,V_k)$ by understanding that $V_i <_\lambda V_j$ iff  
$i<j$.  
\item A \emph{monotone partition} is an ordered partition $(\pi,\lambda)$ with $\pi\in \NC(n)$ such that, for $V,W \in \pi$,  
$V>_\lambda W$ whenever $V$ is an inner block of $W$.    

\item An ordered partition $(\pi,\lambda)$ is \emph{irreducible} if $\pi$ is irreducible. Let $\irr\MP(n)$ denote the set of irreducible monotone partitions. 
\end{enumerate}
\end{defi}
Positivity of random variables or states is irrelevant in this paper,
our treatment is purely algebraic. One reason for this is the  introduction of
a formal random variable $\widetilde{\mathbf{X}}$ in (\ref{eqt}) whose
positivity is not guaranteed. 
Let $(\mathcal{A},\varphi)$ be a pair of a unital algebra over $\comp$ and a unital linear functional on $\mathcal{A}$, i.e.\ $\varphi(1_{\mathcal{A}})=1$.

Let $A_n$ (resp., $a_n$) be one of the cumulant functionals $B_n$, $H_n$, $R_n$ (resp., $b_n$, $h_n$,
$r_n$). 
Given a partition $\pi\in\SP(n)$ and $X, X_i\in\mathcal{A}$, we define  
the associated multivariate and univariate \emph{partitioned cumulant functionals}
\[
A_\pi(X_1,\ldots,X_n):=\prod_{V\in \pi}A_{\abs{V}}(X_V), 
\qquad
a_\pi(X):= A_\pi(X,\ldots,X) = \prod_{V \in \pi} a_{\abs{V}}(X), 
\]
where we use the notation
$$
A_{\abs{V}}(X_V):= A_m(X_{v_1}, \ldots, X_{v_m})
$$  
for a block $V=\{v_1,\ldots,v_m\}$, $v_1<\cdots<v_m$. 
The linear functional $\varphi$ gives rise to the multilinear functional 
$$
(X_1,\dots,X_n)\mapsto \varphi(X_1\cdots X_n)
$$
on $\mathcal{A}^n$ for each $n$ and $\varphi_\pi$ is defined analogously. 

The following formulas implicitly define the classical, free, Boolean
and monotone cumulants.  
\begin{thm}
\begin{align}
\varphi_\pi(X_1,\dotsm, X_n)&=\sum_{\substack{\sigma \in \SP(n) \\ \sigma \leq
    \pi}}K_{\sigma}(X_1,\ldots,X_n),
&\text{\cite{Schutzenberger:1947:certains,Rota:1984:foundations1}}\label{mcclassical}\\
\varphi_\pi(X_1,\dotsm, X_n)&=\sum_{\substack{\sigma \in \NC(n) \\ \sigma \leq \pi}}R_{\sigma}(X_1,\ldots,X_n), &\text{\cite{S94}}\label{mcfree}\\ 
\varphi_\pi(X_1,\dotsm, X_n)&=\sum_{\substack{\sigma \in \I(n) \\ \sigma \leq \pi}}B_{\sigma}(X_1,\ldots,X_n), &\text{\cite{SW97}}\label{mcboolean}\\
\varphi(X_1\dotsm X_n)&=\sum_{(\sigma,\lambda) \in \MP(n)}\frac{1}{\abs{\sigma}!}H_{\sigma}(X_1,\ldots,X_n). &\text{\cite{HS11b}}\label{mcmonotone}
\end{align}
\end{thm}
The multiplicative extension of the monotone case (\ref{mcmonotone}) is not very useful because the summand would depend on both $\sigma,\pi$ (but if $\pi$ is an interval partition, the summand does not depend on $\sigma$; see the proof of Theorem \ref{mainthm1} in Section \ref{sec:proof1}). 

Let $\mu_{\SP}, \mu_\NC, \mu_{\I}$ be the M\"{o}bius functions on the posets
$\SP,\NC,\I$ respectively. 
The values are
\begin{align}
  \mu_\SP(\hat{0}_n,\hat{1}_n) &= (-1)^{n-1} (n-1)!, 
  &\text{\cite{Schutzenberger:1947:certains,Rota:1984:foundations1}}
  \label{eq:moebius:schutzenberger}
 \\
  \mu_\NC(\hat{0}_n,\hat{1}_n) &= (-1)^{n-1} C_{n-1}, 
  &\text{\cite{Kreweras:1972:partitions}}\\
  \mu_\I(\hat{0}_n,\hat{1}_n) &= (-1)^{n-1}, \label{eq:intervalmoebius}
\end{align}
where $C_n=\frac{1}{n+1}\binom{2 n}{n}$ is the Catalan number.
The values $\mu(\pi,\sigma)$ for general intervals $[\pi,\sigma]$ are products of these
due to the fact that in all lattices considered here any such interval is
isomorphic to a direct product of full lattices of different orders, see
\cite{DoubiletRotaStanley:1972:foundations6,S94}.

In fact it is easy to see that the lattice of interval partitions of order $n$
is antiisomorphic to the lattice of subsets of a set with $n-1$ elements
and formula \eqref{eq:intervalmoebius} is equivalent to the inclusion-exclusion
principle.

From the M\"{o}bius principle we may express the classical, free and Boolean cumulants as 
\begin{align}
K_\pi(X_1,\dotsm, X_n)&=\sum_{\substack{\sigma \in \SP(n) \\ \sigma \leq \pi}}
     \varphi_{\sigma}(X_1,\ldots,X_n)
     \, \mu_\SP(\sigma,\pi), \label{cmclassical} \\
R_\pi(X_1,\dotsm, X_n)&=\sum_{\substack{\sigma \in \NC(n) \\ \sigma \leq \pi}}
     \varphi_{\sigma}(X_1,\ldots,X_n)
     \,\mu_\NC(\sigma,\pi), \label{cmfree}\\
B_\pi(X_1,\dotsm, X_n)&=\sum_{\substack{\sigma \in \I(n) \\ \sigma \leq \pi}}
     \varphi_{\sigma}(X_1,\ldots,X_n)
     \, \mu_\I(\sigma,\pi). \label{cmboolean}
\end{align}

Alternatively, univariate cumulants can be defined via generating functions as follows.
Let $(m_n)_{n\geq 1}$  be a sequence with $m_0=1$ and $F(z) = \sum_{n=0}^\infty
\frac{m_n}{n!}z^n$ its exponential generating function and $M(z) =
\sum_{n=0}^\infty m_nz^n$ the ordinary generating 
function. 
\begin{enumerate}
 \item 
  The exponential generating function of the classical cumulants satisfies the
  identity
  $$
  \sum_{n=1}^\infty \frac{\kappa_n}{n!}\,z^n  = \log F(z). 
  $$
 \item The ordinary generating function of the free cumulants
  $$
  R(z) = \sum_{n=1}^\infty r_n z^n
  $$
  is called \emph{$R$-transform} and satisfies the equivalent identities
  \begin{align}
    \label{eq:1+R(zM)=M}
    1 + R(zM(z)) &=  M(z), \\
    \label{eq:M(z/(1+R))=1+R}
    M(z/(1+R(z))) &= 1 + R(z). 
  \end{align}
 \item The ordinary generating function of the Boolean cumulants
  $$
  B(z) = \sum_{n=1}^\infty b_n z^n
  $$
  satisfies the identity
  \begin{equation}
    \label{eq:M=1/1-B}
  M(z) = \frac{1}{1-B(z)}. 
  \end{equation}
\end{enumerate}

Our proofs make use of both the set partition machinery and multivariate
versions of the following generating function relations.
Consider the following identities: 
\begin{equation}
\label{eq:swapidentities}
1 + R
\left(
  \frac{z}{1-B(z)}
\right)
=\frac{1}{1-B(z)}, 
\qquad  
1- B
\left(
  \frac{z}{1+R(z)}
\right)
=\frac{1}{1+R(z)}. 
\end{equation}
The left hand identity is obtained by substituting \eqref{eq:M=1/1-B} into
\eqref{eq:1+R(zM)=M};
and the right hand identity follows from
\eqref{eq:M(z/(1+R))=1+R} by taking reciprocals on both sides
and substituting $1/M(z) = 1-B(z)$.

Define a map $~\tilde{}~$ on the set of generating functions
by setting $\tilde{R}=-B$,
$\tilde{B}=-R$. This map swaps the identities
\eqref{eq:swapidentities}
and  explains the 
occurrence of the factor $(-1)^{\abs{\pi}-1}$ in  formulas
(\ref{eq:boolean2free:BeNi1}),(\ref{mf}) and (\ref{eq:bm}). 
It turns out that $\tilde{H}=-H$ under this transformation, 
where $H(z)$ is the generating function of monotone cumulants.
The details and the multivariate generalization of this observation
are worked out in Section \ref{sec:proof1}.

\section{Counting monotone partitions}\label{sec:monotonepartition}
We want to count the number of monotone labellings of a noncrossing partition $\pi$, i.e., the number of possible orders $\lambda$ on the blocks of $\pi$ such that $(\pi,\lambda)$ becomes a monotone partition. For this purpose 
it is convenient to map the nesting structure of noncrossing partitions
to trees. 

\begin{defi}
  \label{def:nestingforest}
  The \emph{nesting forest} $\nestforest(\pi)$ 
  of a noncrossing partition $\pi$ with $k$ blocks
  is the forest of planar rooted trees on $k$ vertices
  built recursively as follows.
  \begin{enumerate}
   \item 
    If $\pi$ is an irreducible partition, 
    then $\nestforest(\pi)$ is the planar rooted tree, whose
    vertices are the blocks of $\pi$, the root being the unique outer block,
    and branches $\nestforest(\pi_i)$ where $\pi_i$ are the irreducible components
    of $\pi$ without the outer block.
   \item   
    If $\pi$ has irreducible components $\pi_1,\pi_2,\dots,\pi_k$,
    then $\nestforest(\pi)$ is the forest consisting of the rooted trees
    $\nestforest(\pi_1),\nestforest(\pi_2),\dots,\nestforest(\pi_k)$.
  \end{enumerate}
\end{defi}

\begin{figure}[h]
  \centering
\tikzstyle{every node}=[circle, draw=red, fill=red,
                        inner sep=0pt, minimum width=2pt]
\begin{tikzpicture}[scale=0.07]
  \draw (2,0)--(2,9.95);
  \draw (8,0)--(8,9.95);
  \draw (14,0)--(14,6.65);
  \draw (20,0)--(20,3.35);
  \draw (26,0)--(26,3.35);
  \draw (32,0)--(32,6.65);
  \draw (38,0)--(38,3.35);
  \draw (44,0)--(44,3.35);
  \draw (50,0)--(50,3.35);
  \draw (56,0)--(56,9.95);
  \draw (62,0)--(62,9.95);
  \draw (68,0)--(68,3.35);
  \draw (74,0)--(74,3.35);
  \draw (80,0)--(80,9.95);
  \draw (86,0)--(86,6.65);
  \draw (92,0)--(92,3.35);
  \draw (98,0)--(98,6.65);
  \draw (104,0)--(104,9.95);
  \draw (20,3.35)--(26,3.35);
  \draw (14,6.65)--(32,6.65);
  \draw (38,3.35)--(38,3.35);
  \draw (44,3.35)--(50,3.35);
  \draw (2,9.95)--(56,9.95);
  \draw (68,3.35)--(74,3.35);
  \draw (92,3.35)--(92,3.35);
  \draw (86,6.65)--(98,6.65);
  \draw (62,9.95)--(104,9.95);
    \draw [dotted,color=red] (23,6.65)--(29,9.95);
    \draw [dotted,color=red] (23,3.35)--(23,6.65);
    \draw [dotted,color=red] (38,3.35)--(29,9.95);
    \draw [dotted,color=red] (47,3.35)--(29,9.95);
    \draw [dotted,color=red] (71,3.35)--(83,9.95);
    \draw [dotted,color=red] (92,6.65)--(83,9.95);
       \draw [dotted,color=red] (92,3.35)--(92,6.65);
  \node at (29,9.95)  {};
  \node at (23,6.65) {};
  \node at (23,3.35) {};
  \node at (38,3.35) {};
  \node at (47,3.35) {};
  \node at (83,9.95) {};
  \node at (71,3.35) {};
  \node at (92,6.65) {};
  \node at (92,3.35) {};
\end{tikzpicture}
\qquad{}
\tikzset{
  every text node part/.style = {font=\scriptsize},
  grow=up,
  inner sep=0.2mm,
  vertex/.style={circle,draw},
  emptyvertex/.style={},
  level distance=3ex,
  level 1/.style={sibling distance=4em},
  level 2/.style={sibling distance=2em}
}
\tikzstyle{every node}=[circle, draw=black, fill=black,
                        inner sep=0pt, minimum width=2pt]
\begin{tikzpicture}
  [grow=down,inner sep=0mm]
  \node  (root) {}
      child{ node  {} 
        child {node{}}
      }
      child{ node  {} }
      child{ node  {} };
\end{tikzpicture}    
\begin{tikzpicture}
  [grow=down,inner sep=0mm]
    \node  (root) {}
      child{ node {} }
      child{ node {}
        child{node {}}
    };
\end{tikzpicture}

  \caption{A noncrossing partition and its nesting forest}
  \label{fig:nestingtree}
\end{figure}
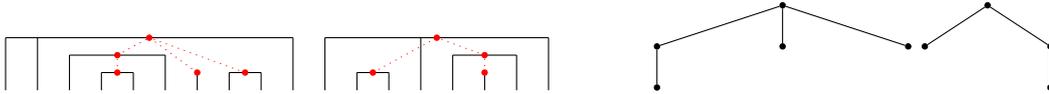


Every monotone labelling of the noncrossing partition $\pi$
corresponds to an increasing labelling of its nesting forest $\tau(\pi)$.
For the enumeration of the latter the so-called \emph{tree factorial} is useful.

\begin{defi}
  The \emph{tree factorial} $t!$ of a finite rooted tree $t$ is recursively defined as follows.
  Let $t$ be a rooted tree with $n>0$ vertices.
  If $t$ consists of a single vertex, set $t!=1$.
  Otherwise $t$ can be decomposed into its root vertex
  and branches $t_1$, $t_2$, \ldots, $t_r$
  and we define recursively  the number
  $$
  t! = n\cdot t_1!\, t_2!\,\dotsm t_r!. 
  $$
  The tree factorial of a forest is the product of the factorials
  of the constituting trees.
\end{defi}

\begin{prop}
  \label{prop:mpi=nestforest}
  \begin{enumerate}[(a)]
   \item 
    The number $m(\pi)$ of monotone labellings of a noncrossing partition $\pi$
    depends only on its nesting forest $\tau(\pi)$ and is given by
    $$
    m(\pi)
    =
    \frac{\abs{\pi}!}{\nestforest(\pi)!}. 
    $$
   \item\label{2A}
    The function
    $w(\pi)= \frac{m(\pi)}{\abs{\pi}!}=\frac{1}{\nestforest(\pi)!}$
    is multiplicative, i.e., if $\pi$ has irreducible
    components $\pi_1,\pi_2,\dots,\pi_k$, then
    $$
    w(\pi) = w(\pi_1) \, w(\pi_2) \dotsm w(\pi_k). 
    $$
  \end{enumerate}
\end{prop}

\begin{proof}
  \begin{enumerate}[(a)]
   \item 
  Here is an inductive proof of the  well known fact
  that the number of increasing labellings of a tree $t$
  is   equal to $\abs{t}!/t!$.
  
  If the tree has only one vertex the claim is obviously true.
  So assume that there are at least $2$ vertices.
  The root must get the smallest label, so there is no choice.
  Then we have to distribute the remaining labels among
  the branches $t_1$, $t_2$, \ldots{},$t_r$. There are 
  $$
  \binom{\abs{t}-1}{\abs{t_1},\abs{t_2},\dots,\abs{t_r}}
  $$
  possibilities to do so and by induction on each branch $t_i$ there are
  $\frac{\abs{t_i}!}{t_i!}$ monotone labellings.
  Putting these together we obtain 
  $$
  \binom{\abs{t}-1}{\abs{t_1},\abs{t_2},\dots,\abs{t_r}}
  \frac{\abs{t_1}!}{t_1!}
  \frac{\abs{t_2}!}{t_2!}
  \dotsm
  \frac{\abs{t_r}!}{t_r!}
  = \frac{1}{\abs{t} } \frac{\abs{t}!}{t_1!t_2!\dotsm t_r!} = \frac{\abs{t}!}{t!}. 
  $$
 \item is immediate from the definition of the tree factorial.
  \end{enumerate}
\end{proof}
From this we can rewrite the formula expressing moments in terms of monotone cumulants (\ref{mcmonotone}) into 
\begin{equation}
\varphi(X_1\cdots X_n)=\sum_{\sigma \in\NC(n)} \frac{1}{\tau(\sigma)!} H_\sigma (X_1,\dots,X_n). 
\end{equation}

We can count the number of monotone partitions as follows. 
Let $\mathit{IB}(n, k)$ be a subset of $\mathit{IB}(n)$ defined by $\{V \in \mathit{IB}(n); \abs{V}=k \}$.  We notice that $\abs{\mathit{IB}(n,k)} = n-k
+1$.  \begin{prop}
$\abs{\mathcal{M}(n)} = \frac{(n+1)!}{2}$. 
\end{prop}
\begin{proof}
There is a bijection $\Phi: \mathcal{M}(n) \to \Big(\bigcup_{k = 1} ^{n-1}
\mathcal{M}(k)\times \mathit{IB}(n, n-k)\Big) \cup \{\hat{1}_n \}$ defined by
\[
\Phi: (V_1,  \ldots, V_{\abs{\pi}}) \mapsto ((V_1, \ldots, V_{\abs{\pi} - 1}), V_{\abs{\pi}}). 
\] 
Now let $a_n : = \abs{\mathcal{M}(n)}$. In view of the above bijection $\Phi$, we have $a_n = 1 + \sum_{k=1} ^{n-1}  (k + 1)a_k$. 
By calculating $a_n-a_{n-1}$, one gets the relation $a_n = (n+1) a_{n-1}$ and so the result follows.
\end{proof}

\section{Symmetries of generating functions and proof of Theorem \ref{mainthm1}}\label{sec:proof1}

\begin{proof}[Proof of Theorem \eqref{mb}] The proof follows the same line as the proof of
(\ref{eq:free2boolean}) in \cite{L02} by a simple application
of the principle of M\"obius inversion (Prop.~\ref{prop:moebius}).

We know that the Boolean cumulants are uniquely determined by the property that 
$$
\varphi_\pi(X_1,X_2,\dots,X_n) = \sum_{\substack{\rho\in\I(n)\\
\rho\leq \pi}} B_\rho(X_1,X_2,\dots,X_n)
$$
for every $\pi\in \I(n)$. We define
$$
\hat{B}_n = \sum_{\pi\in \irr\MP(n)} \frac{1}{\abs{\pi}!}H_\pi
= \sum_{\pi\in \irr\NC(n)} \frac{1}{\nestforest(\pi)!}H_\pi
$$
and show that 
$$
\varphi_\pi(X_1,X_2,\dots,X_n) = \sum_{\substack{\rho\in\I(n)\\
\rho\leq \pi}} \hat{B}_\rho(X_1,X_2,\dots,X_n), 
$$
which then implies $\hat{B}_\pi=B_\pi$ for all $\pi \in\I(n)$ by the M\"obius principle. 
 
First note that multiplicativity of the nesting tree factorial
(Proposition \ref{prop:mpi=nestforest}(\ref{2A})) implies that
$$
\hat{B}_\rho
= \sum_{\substack{\pi\in \MP(n) \\ \hat{\pi}=\rho}} \frac{1}{\abs{\pi}!}H_\pi
= \sum_{\substack{\pi\in \NC(n)\\ \hat{\pi}=\rho}} \frac{1}{\nestforest(\pi)!}H_\pi. 
$$
any interval partition $\rho$.
Moreover given an interval partition $\pi=\{V_1,V_2,\dots,V_p\}$
we have
\begin{align*}
\varphi_\pi&(X_1,X_2,\dots,X_n) \\
&= \sum_{\sigma_1\in \MP(V_1)}
   \sum_{\sigma_2\in \MP(V_2)}
   \dotsm
   \sum_{\sigma_p\in \MP(V_p)}
   \frac{1}{\abs{\sigma_1}!} H_{\sigma_1}(X_{V_1})
   \frac{1}{\abs{\sigma_2}!} H_{\sigma_2}(X_{V_2})
   \dotsm
   \frac{1}{\abs{\sigma_p}!} H_{\sigma_p}(X_{V_p})
\\
&= \sum_{\sigma_1\in \NC(V_1)}
   \sum_{\sigma_2\in \NC(V_2)}
   \dotsm
   \sum_{\sigma_p\in \NC(V_p)}
   \frac{1}{\nestforest(\sigma_1)!} H_{\sigma_1}(X_{V_1})
   \frac{1}{\nestforest(\sigma_2)!} H_{\sigma_2}(X_{V_2})
   \dotsm
   \frac{1}{\nestforest(\sigma_p)!} H_{\sigma_p}(X_{V_p})
\\ 
&= \sum_{\substack{\sigma\in \NC(n)\\ \sigma\leq \pi}}
   \frac{1}{\nestforest(\sigma)!} H_{\sigma}(X_1,X_2,\dots,X_n)
\\ 
&= \sum_{\substack{\rho\in \I(n)\\ \rho\leq \pi }}
   \sum_{\substack{\sigma\in \NC(n)\\ \hat{\sigma}=\rho } }
   \frac{1}{\nestforest(\sigma)!} H_{\sigma}(X_1,X_2,\dots,X_n)
\\ 
&= \sum_{\substack{\rho\in \I(n)\\ \rho\leq \pi}}
   \hat{B}_\rho(X_1,X_2,\dots,X_n), 
\end{align*}
where the multiplicativity of the nesting tree factorial $\tau(\sigma)!$ was
used again for the third equality.

\end{proof}

We are going to show (\ref{mf}) and (\ref{eq:boolean2free:BeNi1}) by using generating functions. The latter was shown in \cite{BN08}, but our proof is different. 
Let $\cseries$
be the ring of formal power series on $r$ free indeterminates $z_1,\ldots,z_r$. Let $\mathbf{z}$ denote the vector $(z_1,\ldots, z_r)$. 
For a vector of noncommutative elements $\mathbf{X}=(X_1, \ldots, X_r)$, we introduce generating functions: 
\begin{align}
M_{\mathbf{X}} (\mathbf{z}) &:= 1 + \sum_{n=1}^\infty \sum_{i_1, \ldots, i_n = 1} ^r \varphi(X_{i_1}\dotsm X_{i_n})z_{i_1}\dotsm z_{i_n}, \\
B_{\mathbf{X}} (\mathbf{z}) &:= \sum_{n=1}^\infty \sum_{i_1, \ldots, i_n = 1} ^r B_n(X_{i_1},\ldots, X_{i_n})z_{i_1}\dotsm z_{i_n}, \\
R_{\mathbf{X}} (\mathbf{z}) &:= \sum_{n=1}^\infty \sum_{i_1, \ldots, i_n = 1} ^r R_n(X_{i_1},\ldots, X_{i_n})z_{i_1}\dotsm z_{i_n}, \\
H_{\mathbf{X}} (\mathbf{z}) &:= \sum_{n=1}^\infty \sum_{i_1, \ldots, i_n = 1} ^r H_n(X_{i_1},\ldots, X_{i_n})z_{i_1}\dotsm z_{i_n}. 
\end{align}
We also introduce vectors of generating functions: 
\begin{align}
\mathfrak{M}_{\mathbf{X}}(\mathbf{z})&:= \mathbf{z}M_{\mathbf{X}}(\mathbf{z})=(z_1M_{\mathbf{X}}(\mathbf{z}),\ldots,z_r M_{\mathbf{X}}(\mathbf{z})),\\ 
\mathfrak{H}_{\mathbf{X}}(\mathbf{z})&:= \mathbf{z}H_{\mathbf{X}}(\mathbf{z})=(z_1H_{\mathbf{X}}(\mathbf{z}),\ldots,z_r H_{\mathbf{X}}(\mathbf{z})). 
\end{align}

\begin{lem} The following identities hold.  
\begin{align}
B_{\mathbf{X}}(\mathbf{z}) M_{\mathbf{X}}(\mathbf{z}) &=M_{\mathbf{X}}(\mathbf{z})-1, \label{B}\\
R_{\mathbf{X}} (\mathfrak{M}_{\mathbf{X}}(\mathbf{z}))&= M_{\mathbf{X}}(\mathbf{z})-1. \label{eq:F}
\end{align}
\end{lem}
\begin{proof}
The first identity is a straightforward consequence of the Boolean moment-cumulant formula. One has 
\begin{multline*}
B_{\mathbf{X}}(\mathbf{z})M_{\mathbf{X}}(\mathbf{z})\\
=\left( \sum_{p=1}^\infty \sum_{i_1, \ldots, i_p = 1} ^r B_p(X_{i_1},\ldots, X_{i_p})z_{i_1}\dotsm z_{i_p}\right)
\left(1+\sum_{q=1}^\infty \sum_{j_1, \ldots, j_q = 1} ^r \varphi(X_{j_1}\dotsm X_{j_q})z_{j_1}\dotsm z_{j_q}\right). 
\end{multline*}
On the right hand side, the coefficient of the term $z_{k_1}\dotsm z_{k_n}$ is equal to 
$$
B_n(X_{k_1},\ldots,X_{k_n})
 +\sum_{p=1}^{n-1} B_p(X_{k_1},\ldots,X_{k_p})\,
                 \varphi(X_{k_{p+1}}\dotsm X_{k_n}), 
$$
which, by virtue of the Boolean moment-cumulant formula, equals 
\begin{multline*}
B_n(X_{k_1},\ldots,X_{k_n})+\sum_{p=1}^{n-1}\sum_{\pi\in\I(n-p)}
B_p(X_{k_1},\ldots,X_{k_p}) B_\pi(X_{k_{p+1}},\ldots,X_{k_n})\\
\begin{aligned}
&= B_n(X_{k_1},\ldots,X_{k_n})+\sum_{\sigma\in\I(n),\sigma\neq\hat{1}_n} B_\sigma(X_{k_{1}},\ldots,X_{k_n})\\
&=\varphi(X_{k_1}\dotsm X_{k_n}). 
\end{aligned}
\end{multline*}
Formula \eqref{eq:F} is the fundamental identity defining the multivariate
$R$-transform, see \cite[Lecture~16]{NS06}. 
\end{proof}
For $t \in \real$ and $\mathbf{X}=(X_1,\ldots,X_r)$, let $\mathbf{X}(t)=(X_1(t), \ldots,X_r(t))$  be a vector whose joint distribution is characterized by 
\begin{equation}\label{process} 
H_n(X_{i_1}(t), \ldots, X_{i_n}(t))= t H_n(X_{i_1}, \ldots, X_{i_n}),\qquad i_1,\ldots,i_n \in[r],~ n\geq1, 
\end{equation}
or equivalently 
$$
\varphi(X_{i_1}(t)\dotsm X_{i_n}(t))
= \sum_{\pi \in \mathcal{M}(n)} \frac{t^{\abs{\pi}}}{\abs{\pi}!} H_\pi(X_{i_1}, \ldots, X_{i_n}). 
$$
Let us denote 
\begin{align*}
M_{\mathbf{X}}(t,\mathbf{z})&:=M_{\mathbf{X}(t)}(\mathbf{z}), \\
\mathfrak{M}_{\mathbf{X}}(t, \mathbf{z})&:= \mathbf{z}M_{\mathbf{X}}(t,\mathbf{z})= (z_1 M_{\mathbf{X}}(t,\mathbf{z}),\ldots,z_r M_{\mathbf{X}}(t,\mathbf{z})).
\end{align*} Then the following differential equation holds as formal power series \cite{HS11b}: 
\begin{equation}
\label{eq:monotonediffeq}
\frac{d}{dt}\mathfrak{M}_{\mathbf{X}}(t,\mathbf{z})
= \mathfrak{H}_{\mathbf{X}}(\mathfrak{M}_{\mathbf{X}}(t, \mathbf{z})),\qquad \mathfrak{M}_{\mathbf{X}}(0,\mathbf{z})=\mathbf{z}.
\end{equation}
Moreover, $(\mathfrak{M}_{\mathbf{X}}(t,\cdot))_{t\in\real}$ becomes a flow on $\cseries$:
\begin{equation}\label{group}
\mathfrak{M}_{\mathbf{X}}(t+s,\mathbf{z})=\mathfrak{M}_{\mathbf{X}}(t,\mathfrak{M}_{\mathbf{X}}(s,\mathbf{z})),\qquad t,s\in\real, 
\end{equation}
 which is proved by standard techniques from ordinary differential equations using the uniqueness of the solution in $\cseries$. 

\begin{defi}
For a vector $\mathbf{X}=(X_1,\ldots,X_r) \in\mathcal{A}^r$, let $\widetilde{\mathbf{X}}=(\widetilde{X}_1,\ldots, \widetilde{X}_r)$ be a vector satisfying the relation
\begin{equation}\label{eqt}
R_n(\widetilde{X}_{i_1},\ldots, \widetilde{X}_{i_n}) = -B_n(X_{i_1}, \ldots, X_{i_n}) 
\end{equation}
for any tuple $(i_1, \ldots,i_n)\in[r]^n$. 
\end{defi}

\begin{lem} \label{lem1}
The following relations hold for any $\mathbf{X}$ and any tuple $(i_1, \ldots,i_n)$: 
\begin{enumerate}[\rm(1)]
\item $B_n(\widetilde{X}_{i_1},\ldots, \widetilde{X}_{i_n}) = -R_n(X_{i_1}, \ldots, X_{i_n})$, or equivalently $B_{\widetilde{\mathbf{X}}}(\mathbf{z})=- R_{\mathbf{X}}(\mathbf{z})$;  
\item $H_n(\widetilde{X}_{i_1},\ldots, \widetilde{X}_{i_n}) = -H_n(X_{i_1}, \ldots, X_{i_n})$, or equivalently $H_{\widetilde{\mathbf{X}}}(\mathbf{z})= -H_{\mathbf{X}}(\mathbf{z})$. 
\end{enumerate}
\end{lem}
\begin{proof}
(1)\,\, We will show that 
\begin{equation}\label{tilde} 
\mathfrak{M}_{\mathbf{X}} \circ \mathfrak{M}_{\widetilde{\mathbf{X}}} = \text{Id}.
\end{equation} 
The definition (\ref{eqt}) of $\widetilde{\mathbf{X}}$ reads $R_{\widetilde{\mathbf{X}}}(\mathbf{z})=-B_{\mathbf{X}}(\mathbf{z})$, so that 
\begin{equation}\label{eqt1}
-R_{\widetilde{\mathbf{X}}}(\mathbf{z})M_{\mathbf{X}}(\mathbf{z})=M_{\mathbf{X}}(\mathbf{z})-1
\end{equation}
from (\ref{B}). Replace $\mathbf{z}$ by $\mathfrak{M}_{\widetilde{\mathbf{X}}}(\mathbf{z})$ and then (\ref{eqt1}) becomes 
$$
-\left(M_{\widetilde{\mathbf{X}}}(\mathbf{z})-1\right)M_{\mathbf{X}}(\mathfrak{M}_{\widetilde{\mathbf{X}}}(\mathbf{z}))=M_{\mathbf{X}}(\mathfrak{M}_{\widetilde{\mathbf{X}}}(\mathbf{z}))-1, 
$$
where the relation (\ref{eq:F}) was used for $\widetilde{\mathbf{X}}$. Hence, $z_i M_{\widetilde{\mathbf{X}}}(\mathbf{z})M_{\mathbf{X}}(\mathfrak{M}_{\widetilde{\mathbf{X}}}(\mathbf{z})) = z_i$ for each $i$, implying the claim (\ref{tilde}). In particular, $M_{\mathbf{X}}(\mathfrak{M}_{\widetilde{\mathbf{X}}}(\mathbf{z}))=\frac{1}{M_{\widetilde{\mathbf{X}}}(\mathbf{z})}$. Replacing $\mathbf{z}$ by $\mathfrak{M}_{\widetilde{\mathbf{X}}}(\mathbf{z})$ in (\ref{eq:F}), one obtains $R_{\mathbf{X}}(\mathbf{z})=\frac{1}{M_{\widetilde{\mathbf{X}}}(\mathbf{z})}-1$, which coincides with $-B_{\widetilde{\mathbf{X}}}(\mathbf{z})$ thanks to (\ref{B}) for $\widetilde{\mathbf{X}}$.

(2)\,\, 
The flow property (\ref{group}) for $t=1,s=-1$ reads $\mathfrak{M}_\mathbf{X}\circ \mathfrak{M}_{\mathbf{X}(-1)}=\text{Id}$, which together with (\ref{tilde}) implies that $\mathbf{X}(-1)=\widetilde{\mathbf{X}}$ in distribution regarding $\varphi$. From (\ref{process}) we get 
$$
H_n(\widetilde{X}_{i_1}, \ldots, \widetilde{X}_{i_n})=H_n(X_{i_1}(-1),\ldots, X_{i_n}(-1))=- H_n(X_{i_1}, \ldots, X_{i_n}). 
$$
\end{proof}

\begin{proof}[Proof of Theorem \ref{mainthm1}\eqref{mf}]
From (\ref{mb}) and Lemma~\ref{lem1}, we obtain 
\begin{align*}
-R_n(X_1, \ldots,X_n)&=B_n(\widetilde{X}_1,\ldots,\widetilde{X}_n)= \sum_{\pi\in \irr\MP(n)} \frac{1}{\abs{\pi}!}H_\pi(\widetilde{X}_1,\ldots,\widetilde{X}_n)\\
&=\sum_{\pi\in \irr\MP(n)}\frac{(-1)^{\abs{\pi}}}{\abs{\pi}!}H_\pi(X_1,\ldots,X_n).
\end{align*}
\end{proof}

\begin{proof}[Proof of \eqref{eq:boolean2free:BeNi1}] From the multi-variate version generalization of (\ref{eq:free2boolean}) and Lemma \ref{lem1}, we have the following: 
\begin{align*}
-R_n(X_1, \ldots,X_n)
&=B_n(\widetilde{X}_1,\ldots,\widetilde{X}_n)
 = \sum_{\pi\in \irr\NC(n)} R_\pi(\widetilde{X}_1,\ldots,\widetilde{X}_n)\\
&=\sum_{\pi\in \irr\NC(n)}(-1)^{\abs{\pi}}B_\pi(X_1,\ldots,X_n).
\end{align*}
\end{proof}

\begin{rem} $\irr\NC(n)$ is a lattice and is isomorphic to $\NC(n-1)$;
 however we can not apply the M\"obius inversion  directly to the
 free-to-boolean 
 formula \eqref{eq:free2boolean} to get the boolean-to-free formula
 \eqref{eq:boolean2free:BeNi1} since it does not respect multiplicativity.
\end{rem}

\section{Colored trees and proof of Theorem \ref{mainthm2}}\label{sec:colored}

The concept of colored partitions was introduced by Lenczewski \cite{L12}. 
\begin{defi}
  An \emph{$N$-colored partition} of $[n]$ is a pair $(\pi, f)$, where $\pi=\{V_1, \ldots, V_k\}$ is a partition of $[n]$ and  $f$ is a map from the set
  $\{V_1, \ldots, V_k\}$ to $[N]$. The set of noncrossing $N$-colored partitions of $[n]$ is denoted by $\NC(n,N)$. When $i=f(V)$, we say that $V\in\pi$ is colored by $i$. 
\end{defi}
\begin{rem}
  An ordered partition of $[n]$ is a $\abs{\pi}$-colored partition $(\pi,f)$ of $[n]$ such that every block has a different color.
\end{rem}

We will express monotone cumulants in terms of free cumulants. For this
purpose, we are going to associate a polynomial to a rooted tree.

\begin{defi}
  An $N$-labelling of a graph is a map $f$ from its vertices to $[N]$. A
  labelling $f$ of a rooted tree   is called \emph{nondecreasing} 
  if for every vertex $v$ and every child vertex $u$ of $v$
  we have $f(v)\leq f(u)$.
  An $N$-labelling of a forest is called nondecreasing
  if the labels on each of its trees are nondecreasing.
\end{defi}
\begin{prop}
  Let $t$ be a rooted tree and let $P_t(N)$ be the number of nondecreasing
  $N$-labellings of $t$.
  Then $P_t(N)$ is a polynomial  in $N$ of degree $\leq \abs{t}$ with zero constant term.
  An analogous statement holds for forests.
\end{prop}
\begin{proof}
 It is easy to see that the polynomial $P_t(N)$ of a nesting forest consisting
 of rooted trees $t_1,\dots t_k$ is just $\prod_i P_{t_i}(N)$,
 therefore it suffices to show that $P_t(N)$ is a polynomial without constant
 term for every rooted tree $t$. 
 We proceed by induction. If $t$ has only one vertex,
 then $P_t(N)=N$ satisfies the claimed property.
 Otherwise  $t$ decomposes into the root vertex and branches $t_1',\dots, t_m'$.
 If the root gets label $k$, all other vertices must get labels not smaller
 than $k$ and the number of such nondecreasing labellings is
 $$
 Q(N,k)=\prod_{i=1}^m P_{t_i'}(N-k+1)
 $$
 which by assumption is a polynomial in $N$ and $k$ of degree $\leq \abs{t}-1$ without constant term.
 The number of $N$-labellings can be enumerated recursively by conditioning
 on the label of the root and we obtain
 $$
 P_t(N) = \sum_{k=1}^N Q(N,k)
 .
 $$
 We can now apply Faulhaber's summation formula 
 \begin{equation}
   \label{eq:faulhaber}
 \sum_{k=1}^N k^{n} =
 \frac{B_{n+1}(N+1)-B_{n+1}(0)}{n+1}
 \end{equation}
 to each term  to express $P_t(N)$ in terms of Bernoulli polynomials $B_n(x)$
 defined by their generating function
 $$
 \frac{z e^{x z}}{e^z-1}= \sum_{n=0}^\infty B_n(x) \frac{z^n}{n!}. 
 $$
 It follows that $P_t(N)$ is a polynomial of degree $\leq\abs{t}$
 without constant term.
\end{proof}
\begin{defi}
 Let $\alpha_t$ be the coefficient of the linear term of $P_t$, i.e., 
 $$
 \alpha_t=P_t'(0). 
 $$
  For $\pi \in
\NC(n)$, we define $P_\pi(N):= P_{\tau(\pi)}(N)$ and $\alpha_\pi:=\alpha_{\tau(\pi)}$.
\end{defi}
\begin{exa}
\begin{enumerate}
\item If $t$ is not connected then $\alpha_t=0$.
 \item Let $t$ be a tree consisting of the vertices $\{1,2, \ldots, n+1 \}$ and $n$
  edges, each connecting $1$ and $k$ for $2 \leq k \leq n+1$. The vertex $1$ is
  the root of $t$.  Then $P_t(N)$ is equal to Faulhaber's formula
  \eqref{eq:faulhaber}
  and $\alpha_t = P_t'(0)=
  \frac{1}{n+1}B_{n+1}'(1) = B_{n}(1)$ because of the identity
  $B_k'(x)=kB_{k-1}(x)$. So $\alpha_t=B_{n}(1)$ is the $n$th Bernoulli number.

 \item A tree $t$ has $n$ vertices $\{1,2, \ldots, n\}$ and $n-1$ edges, each
  connecting $k$ and $k+1$ for $1 \leq k \leq n-1$, and $1$ is the root of
  $t$. Then $P_t(x) = \binom{x+n-1}{n}$ and $\alpha_t=\frac{1}{n}$.
\end{enumerate}
\end{exa}

\begin{proof}[Proof of Theorem \ref{mainthm2}]
  Let $J$ be a subset of $[N] \times [N]$ including the diagonal set $\{(j,j)
  \mid j \in [N]\}$.  Lenczewski introduced the concept of strong matricial
  freeness for an array $(X_{ij})_{(i,j)\in J}$ of elements in a
  noncommutative probability space $(\mathcal{A}, (\varphi_{ij})_{(i,j) \in
    J})$, where $(\varphi_{ij})_{(i,j) \in
    J}$  is an array of unital linear functionals. Following \cite{L12}, we assume that
  $\varphi_{ii}=\varphi$, the same linear functional, and $\varphi_{ij}$ do not depend on
  $i$ if $i\neq j$.
  Let $(X_{ij})_{(i,j) \in J}$ be an array of elements with
  low-identical distributions, that is, the moments $\varphi(X_{ij}^n)$ do not
  depend on $j$.

  Given $(\pi, f) \in \NC(n,N)$, we associate a product of free
  cumulants $r_{(\pi,f)}:= r_{(V_1,f)} \dotsm r_{(V_{\abs{\pi}},f)}$ as
  follows. Take a block $V_k$ of $\pi$ with color $i$. If its outer blocks are
  all colored by $i$, we define $r_{(V_k, f)}:=r_{\abs{V_k}}(X_{ii})$. If $V_k$
  has an outer block with another color $j$, then
  $r_{(V_k,f)}:=r_{\abs{V_k}}(X_{ij})$ where $j$ is the color of the outer
  block nearest to $V_k$ whose color $j$ is different from $i$. If a pair
  $(i,j)$ is not an element of $J$, then we understand $r_k(X_{ij})=0$, $k
  =1,2,3,\ldots$.

  We need the following results (see Lemma 7.1 and Theorem 3.1 of \cite{L12}
  and also Proposition 4.1 of \cite{L10}):
\begin{enumerate}[\rm(i)]
\item 
 \label{it:Lenc1}
 \begin{equation}\label{eq:Lenc}
 \varphi\left(\left(\sum_{(i,j) \in J}X_{ij}\right)^n\right) = \sum_{(\pi,f) \in\NC(n,N)} r_{(\pi,f)}. 
 \end{equation}
\item
 \label{it:Lenc2}
 Let $(X_{ij})_{(i,j)\in J}$ be strongly matricially free. If $J=\{(i,j)
 \mid 1 \leq i \leq j \leq N \}$ and the distributions of $X_{ij}$ do not
 depend on $j$, then $X_i:=\sum_{j =1}^i X_{ij}$, $i \in[N]$, are monotonically
 independent.
\end{enumerate}
In statement \eqref{it:Lenc2}, note that $X_i$ has the same distribution as $X_{ii}$
with respect to $\varphi$ as mentioned in the proof of Proposition 4.1 of
\cite{L10}.
In statement \eqref{it:Lenc2} we assume in addition that the distributions of $X_{ij}$ do
not depend on $(i,j)$. Then formula \eqref{eq:Lenc} in statement \eqref{it:Lenc1} becomes
$$
\varphi((X_1+\cdots+X_N)^n)= \sum_{(\pi,f) \in\NC(n,N)} r_{(\pi,f)}
$$ 
for monotonically i.i.d.\ random variables $X_i$. Each summand $r_{(\pi,f)}$ is
either $\prod_{V \in \pi} r_{\abs{V}}(X_1)$ or $0$. The summand $r_{(\pi,f)}$
does not vanish if and only if, for each block $V$ with color $i$, its outer
blocks have colors not greater than $i$. The number of such colorings is just
equal to $P_\pi(N)$. By definition, the $n$th monotone cumulant of $X_1$
coincides with the coefficient of $N$ in $\varphi((X_1+\cdots+X_N)^n)$ 
(which, in particular, is zero unless $\pi\in\irr\NC(n)$), so (\ref{eq:fm}) follows.

The identity (\ref{eq:bm}) is proved similarly to Theorem \ref{mainthm1}(\ref{mf}).  
From (\ref{eq:fm}) and Lemma \ref{lem1}, it follows that
\[
\begin{split}
-h_n(X)&=h_n(\widetilde{X})= \sum_{\pi\in \NC(n)} \alpha_\pi r_\pi(\widetilde{X})=\sum_{\pi\in \NC(n)}(-1)^{\abs{\pi}} \alpha_\pi b_\pi(X).
\end{split}
\]

Identity (\ref{eq:cm}) follows from (\ref{eq:fm}) together with the easy fact
(following from (\ref{eq:class2free})) that more generally
$$
r_\sigma=\sum_{\substack{\pi\in\SP(n)\\\bar{\pi}=\sigma}}\kappa_\pi,\qquad \sigma \in\NC(n). 
$$ 
Indeed, we observe that $\pi\in\irr\SP(n)$ if and only if
$\bar{\pi}\in\irr\NC(n)$. 
Hence 
$$
\sum_{\pi\in \irr\SP(n)}\alpha_{\bar{\pi}} \kappa_\pi=\sum_{\sigma\in
  \irr\NC(n)}\alpha_{\sigma}\sum_{\substack{\pi\in\SP(n)\\\bar{\pi}=\sigma}}\kappa_\pi=\sum_{\sigma\in
  \irr\NC(n)}\alpha_{\sigma}r_\sigma=h_n.
$$
\end{proof}

\begin{rem} 
The moment-cumulant formula (\ref{eq:Lenc}) is known only for the univariate case, and so we can prove Theorem \ref{mainthm2} only for univariate cumulants.  
\end{rem}

\section{Tutte polynomials and proof of Theorem \ref{mainthm3}}\label{sec:proof3}

For an arbitrary finite set $\mathcal{S}$ we denote by $\SP(\mathcal{S})$ its set of partitions. Any bijection between $\mathcal{S}$ and $\{1,\dots \abs{\mathcal{S}}\}$ induces a poset isomorphism $\SP(\mathcal{S})$ to $\SP(\abs{\mathcal{S}})$. If $\mathcal{S}$ is totally ordered we consider the bijection which preserves this order and define $\NC(\mathcal{S}),\I(\mathcal{S})$ via this isomorphism.

\begin{defi}

Let $\pi\in\SP(n)$.
\begin{enumerate}
 \item We define the crossing graph $G(\pi):=(V,E)$ of $\pi$, where the set of
  vertices $V=\{V_1,\dots ,V_{\abs{\pi}}\}$ is indexed by the blocks\footnote{It
    should not cause confusion that we regard $V_i$ simultaneously as a vertex
    of $G(\pi)$ and as a block of $\pi$}
  of $\pi$ and an edge joins the vertices $V_i, V_j$ if and only if they cross,
  i.e., $W=(V_i,V_j)\in(\SP(V_i\cup V_j)\setminus\NC(V_i\cup V_j))$.

 \item Similarly, the vertices of the anti-interval graph
  $\tilde{G}(\pi):=(V, E)$ of $\pi$ are just the blocks of $\pi$. An
  edge joining $(V_i,V_j)$ is drawn if and only if $W=(V_i,V_j)\in(\SP(V_i\cup
  V_j)\setminus\I(V_i\cup V_j))$.
  (For a noncrossing partition this is the nesting forest
  from Definition~\ref{def:nestingforest},
  augmented by the edges from all vertices to all their descendents).

 \item For a finite graph $G=(V,E)$ and $e\in E$, we let $G\setminus
  e=(V,E\setminus e)$, and $G/ e=(V/e,E\setminus e)$ be the graph obtained from
  removing $e$ and identifying the endpoints of $e$. The \emph{Tutte
    polynomial} 
  $T_G(x,y)$ of $G$ can be defined recursively by setting $T_G(x,y)=1$ if
  $E=\emptyset$ and:
  $$
  T_G(x,y)=
  \begin{cases}
    xT_{G/e}(x,y) & \text{if $e$ is a bridge,} \\
    yT_{G\setminus e}(x,y) & \text{if $e$ is a loop,}\\
    T_{G/e}(x,y)+T_{G\setminus e}(x,y) & \text{otherwise.}
  \end{cases}
  $$
\end{enumerate}
\end{defi}

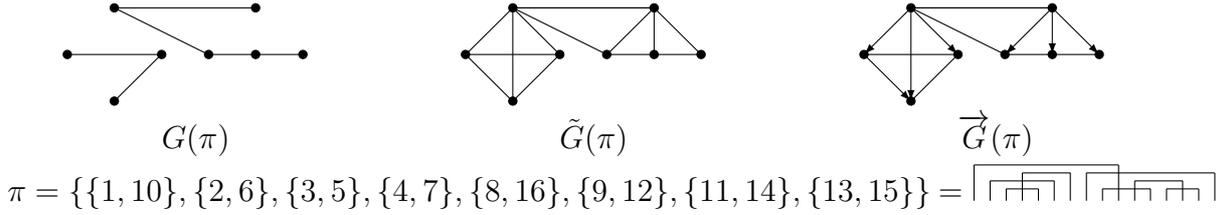
\begin{figure}
    \centering
    \begin{minipage}{0.3\linewidth}
    \setlength{\unitlength}{1.5em}
    \begin{picture}(6.0,4.0)    
      \put(2,1){\circle*{0.2}}
      \put(2,1){\line(1,1){1.0}}
      \put(1,2){\circle*{0.2}}
      \put(3,2){\line(-1,0){2.0}}
      \put(3,2){\circle*{0.2}}
      \put(4,2){\circle*{0.2}}
      \put(4,2){\line(1,0){2.0}}
      \put(5,2){\circle*{0.2}}
      \put(6,2){\circle*{0.2}}
      \put(2,3){\circle*{0.2}}
      \put(2,3){\line(2,-1){2.0}}
      \put(2,3){\line(1,0){3.0}}
      \put(5,3){\circle*{0.2}}
      \put(3,0){$G(\pi)$}
    \end{picture}

    \end{minipage}
    \begin{minipage}{0.3\linewidth}
    \setlength{\unitlength}{1.5em}
    \begin{picture}(6.0,4.0)    
      \put(2,1){\circle*{0.2}}
      \put(2,1){\line(1,1){1.0}}
      \put(1,2){\circle*{0.2}}
      \put(1,2){\line(1,-1){1.0}}
      \put(3,2){\line(-1,0){2.0}}
      \put(3,2){\circle*{0.2}}
      \put(4,2){\circle*{0.2}}
      \put(4,2){\line(1,0){2.0}}
      \put(5,2){\circle*{0.2}}
      \put(6,2){\circle*{0.2}}
      \put(2,3){\circle*{0.2}}
      \put(2,3){\line(-1,-1){1.0}}
      \put(2,3){\line(1,-1){1.0}}
      \put(2,3){\line(0,-1){2.0}}
      \put(2,3){\line(2,-1){2.0}}
      \put(2,3){\line(1,0){3.0}}
      \put(2,3){\line(1,0){3.0}}
      \put(5,3){\circle*{0.2}}
      \put(5,3){\line(-1,-1){1.0}}
      \put(5,3){\line(1,-1){1.0}}
      \put(5,3){\line(0,-1){1.0}}
      \put(3,0){$\tilde{G}(\pi)$}
    \end{picture}
    \end{minipage}
\begin{minipage}{0.3\linewidth}
    \setlength{\unitlength}{1.5em}
    \begin{picture}(6.0,4.0)    
      \put(2,1){\circle*{0.2}}
      \put(2,1){\line(1,1){1.0}}
      \put(1,2){\circle*{0.2}}
      \put(1,2){\vector(1,-1){1.0}}
      \put(3,2){\line(-1,0){2.0}}
      \put(3,2){\circle*{0.2}}
      \put(4,2){\circle*{0.2}}
      \put(4,2){\line(1,0){2.0}}
      \put(5,2){\circle*{0.2}}
      \put(6,2){\circle*{0.2}}
      \put(2,3){\circle*{0.2}}
      \put(2,3){\vector(-1,-1){1.0}}
      \put(2,3){\vector(1,-1){1.0}}
      \put(2,3){\vector(0,-1){2.0}}
      \put(2,3){\line(2,-1){2.0}}
      \put(2,3){\line(1,0){3.0}}
      \put(2,3){\line(1,0){3.0}}
      \put(5,3){\circle*{0.2}}
      \put(5,3){\vector(-1,-1){1.0}}
      \put(5,3){\vector(1,-1){1.0}}
      \put(5,3){\vector(0,-1){1.0}}
      \put(3,0){$\overrightarrow{G}(\pi)$}
    \end{picture}
    \end{minipage}

      $$\pi=\{\{1,10\},\{2,6\},\{3,5\},\{4,7\},\{8,16\},\{9,12\},\{11,14\},\{13,15\}\}=
\begin{picture}(92,12.5)(1,0)
  \put(2,0){\line(0,1){13.5}}
  \put(8,0){\line(0,1){7.5}}
  \put(14,0){\line(0,1){4.5}}
  \put(20,0){\line(0,1){10.5}}
  \put(26,0){\line(0,1){4.5}}
  \put(32,0){\line(0,1){7.5}}
  \put(38,0){\line(0,1){10.5}}
  \put(44,0){\line(0,1){10.5}}
  \put(50,0){\line(0,1){4.5}}
  \put(56,0){\line(0,1){13.5}}
  \put(62,0){\line(0,1){7.5}}
  \put(68,0){\line(0,1){4.5}}
  \put(74,0){\line(0,1){4.5}}
  \put(80,0){\line(0,1){7.5}}
  \put(86,0){\line(0,1){4.5}}
  \put(92,0){\line(0,1){10.5}}
  \put(14,4.5){\line(1,0){12}}
  \put(8,7.5){\line(1,0){24}}
  \put(20,10.5){\line(1,0){18}}
  \put(2,13.5){\line(1,0){54}}
  \put(50,4.5){\line(1,0){18}}
  \put(62,7.5){\line(1,0){18}}
  \put(74,4.5){\line(1,0){12}}
  \put(44,10.5){\line(1,0){48}}
\end{picture}
$$
    \caption{A partition and its associated graphs (see Definition \ref{defimonotone} for the third graph).}
    \label{fig:graphs}
  \end{figure}

  \begin{rem}
    Let $(A_i)_{i\in I}$ be a family of sets.
    Recall that the \emph{intersection graph} is the graph with vertex set
    $\{A_i : i \in I\}$ where there is an edge $i\sim j$ if and only if
    $A_i\cap A_j\ne 0$, see \cite{prisner:1998:journey} for more information.
    An \emph{interval graph} is the intersection graph of a family
    of intervals on the real line.
    Coincidentally, the anti-interval graph defined above
    is exactly the interval graph generated by the convex hulls of the blocks
    of $\pi$. 
  \end{rem}

\begin{rem}
Let $G=(V,E)$ be any finite graph. For $\pi\in \SP(V)$ we define $i(E,\pi)$ to be the number of edges in $E$ which connect vertices with both endpoints belonging to the same block of $\pi$.  It was shown in \cite{JV13} that for any graph $G$ and $q\neq 1$ we have:

\begin{equation}\label{eq:graph}
\frac{1}{(q-1)^{\#V-1}}\sum_{\pi\in\SP(V)}q^{i(E,\pi)}\mu_\SP(\pi,1_V)=
\begin{cases}
T_G(1,q) & \text{if $G$ is connected}, \\
0 & \text{otherwise}
\end{cases}
\end{equation}
with the convention that $q^0=1$ for $q=0$.
\end{rem}

We obtain the Boolean-to-classical cumulant formula by following the lines of the proof of (\ref{eq:free2class}) in \cite{JV13}.

\begin{proof}[Proof of Theorem \ref{mainthm3}]
Let $X_1,\dots,X_n\in \mathcal{A}$. Using subsequently the classical and the
Boolean moment-cumulant formulas (\ref{cmclassical}), (\ref{mcboolean}), we obtain
\begin{align*}
K_n(X_1,\dots,X_n)&=\sum_{\pi \in \SP(n)}\varphi_{\pi}(X_1,\dots,X_n)\mu_\SP(\pi,\hat{1}_n)\\
&=\sum_{\substack{\sigma,\pi \in \SP(n) \\ \sigma\preceq\pi}}B_{\sigma}(X_1,\dots,X_n)\mu_\SP(\pi,\hat{1}_n)\\
&=\sum_{\sigma\in \SP(n)}B_{\sigma}(X_1,\dots,X_n)\sum_{\pi\succeq\sigma}\mu_\SP(\pi,\hat{1}_n),
\end{align*}
where, for $\pi \geq \sigma$, we write $\pi\succeq\sigma$ if the restriction $\sigma|_W\in \SP(W)$ to any block $W$ of $\pi$ is an interval partition.

Let us fix $\sigma \in \SP(n)$ and consider its anti-interval graph $\tilde{G}(\sigma)=(V,E)$. There is a one-to-one correspondence $\pi\mapsto \tilde{\pi}$ between partitions $\{\pi:\pi\geq\sigma\}\subset \SP(n)$ and $\SP(V)$: $\pi \geq \sigma$ is obtained by gluing blocks of $\sigma$, and $\tilde{\pi}$ describes which blocks of $\sigma$ are glued together. In view of the formula (\ref{eq:graph}) for $q=0$, we observe that $i(E,\tilde{\pi})=0$ exactly for those $\pi\geq\sigma$ such that $\pi\succeq\sigma$. Furthermore, $\abs{\pi}=\abs{\tilde{\pi}}$  and hence $\mu_\SP(\tilde{\pi},1_V)=(-1)^{\abs{\tilde{\pi}}-1}(\abs{\tilde{\pi}}-1)!=(-1)^{\abs{\pi}-1}(\abs{\pi}-1)!=\mu_\SP(\pi,\hat{1}_n)$. Therefore 
$$\sum_{\pi\succeq\sigma}\mu_\SP(\pi,\hat{1}_n)=\sum_{\pi\geq\sigma}0^{i(E,\tilde{\pi})}\mu_\SP(\pi,\hat{1}_n)=\sum_{\tilde{\pi}\in\SP(V)}0^{i(E,\tilde{\pi})}\mu_\SP(\tilde{\pi},1_V)$$
and thus, formula (\ref{eq:graph}) yields 
$$
\sum_{\pi\succeq\sigma}\mu_\SP(\pi,\hat{1}_n)=\sum_{\tilde{\pi}\in\SP(V)}0^{i(E,\tilde{\pi})}\mu_\SP(\tilde{\pi},1_V)
=
\begin{cases}
(-1)^{\abs{\sigma}-1}T_{\tilde{G}(\sigma)}(1,0) & \text{if $\tilde{G}(\sigma)$ is connected,} \\
0 & \text{otherwise}.
\end{cases}
$$
It is not hard to see that the number of blocks of the interval closure of $\sigma$ is equal to the number of connected components of its anti-interval graph. Therefore $\tilde{G}(\sigma)$ is connected iff $\sigma\in\irr\SP(n)$ and the formula follows.
\end{proof}

Several evaluations of the Tutte polynomial have combinatorial 
interpretations. For our purposes it will be interesting
to note the fact that
$T_G(1,0)$ equals the number of rooted acyclic orientations
with unique specified source; this number does not depend on the choice
of the source \cite{GreeneZaslavsky:1983:interpretation}.
Recall that an \emph{acyclic orientation} of a graph is an orientation without
oriented cycles. An acyclic orientation has source $v$ if there is a directed
path from $v$ to every other vertex. Alternatively, this can be interpreted
as the Hasse diagram of an ordering of the vertices of $G$ with prescribed
unique minimal element.

Returning from graphs to partitions this has the following  pictorial interpretation.
\begin{defi}
  \begin{enumerate}
   \item   Let $\pi$ be a connected set partition.
    A \emph{crossing heap} on $\pi$ is a poset structure on the blocks of $\pi$
    such that any pair of crossing blocks is comparable.
    A \emph{pyramid} is a crossing heap such that the first block is
    the only maximal element.
   \item Let $\pi$ be an irreducible set partitions.
    An \emph{interval heap} on $\pi$ is a poset structure on the blocks of
    $\pi$
    such that any pair of blocks whose convex hulls have nonempty intersection
    are comparable.
    A \emph{pyramid} is an interval heap such that the first block is
    the only maximal element.
  \end{enumerate}
\end{defi}
The following proposition is immediate from the definition, see
Fig.~\ref{fig:crossingheaps} and Fig.~\ref{fig:intervalheaps}.
\begin{figure}
\begin{minipage}{0.3\linewidth}
  \begin{center}
    \setlength{\unitlength}{3em}
    \begin{picture}(3.0,3.0)(-0.5,0.2)
      \put(1,2){\circle*{0.1}}
      \put(1.1,2.1){$1$}
      \put(0,1){\circle*{0.1}}
      \put(-0.2,1.1){$2$}
      \put(2,1){\circle*{0.1}}
      \put(2.1,1.1){$3$}
      \put(1,0){\circle*{0.1}}
      \put(1.2,-0.2){$4$}
      \put(1,2){\line(1,-1){1}}
      \put(1,2){\line(-1,-1){1}}
      \put(1,0){\line(-1,1){1}}
      \put(1,0){\line(1,1){1}}
    \end{picture}

$$
\setlength{\unitlength}{\standardunit}
G(\begin{picture}(44,9.5)(1,0)
  \put(2,0){\line(0,1){4.5}}
  \put(8,0){\line(0,1){10.5}}
  \put(14,0){\line(0,1){7.5}}
  \put(20,0){\line(0,1){4.5}}
  \put(26,0){\line(0,1){4.5}}
  \put(32,0){\line(0,1){7.5}}
  \put(38,0){\line(0,1){10.5}}
  \put(44,0){\line(0,1){4.5}}
  \put(2,4.5){\line(1,0){18}}
  \put(14,7.5){\line(1,0){18}}
  \put(8,10.5){\line(1,0){30}}
  \put(26,4.5){\line(1,0){18}}
\end{picture})
$$
  \end{center}
\end{minipage}
\begin{minipage}{0.3\linewidth}
  \begin{center}
    \setlength{\unitlength}{3em}
    \begin{picture}(3.0,3.0)(-0.5,0.2)
      \put(1,2){\circle*{0.1}}
      \put(1.1,2.1){$1$}
      \put(0,1){\circle*{0.1}}
      \put(-0.2,1.1){$2$}
      \put(2,1){\circle*{0.1}}
      \put(2.1,1.1){$3$}
      \put(1,0){\circle*{0.1}}
      \put(1.2,-0.2){$4$}
      \put(1,2){\vector(1,-1){0.95}}
      \put(1,2){\vector(-1,-1){0.95}}
      \put(1,0){\vector(-1,1){0.95}}
      \put(1,0){\vector(1,1){0.95}}
    \end{picture}

$$
\setlength{\unitlength}{\standardunit}
\begin{picture}(44,9.5)(1,0)
  \put(2,0){\line(0,1){10.5}}
  \put(8,0){\line(0,1){7.5}}
  \put(14,0){\line(0,1){4.5}}
  \put(20,0){\line(0,1){10.5}}
  \put(26,0){\line(0,1){10.5}}
  \put(32,0){\line(0,1){4.5}}
  \put(38,0){\line(0,1){7.5}}
  \put(44,0){\line(0,1){10.5}}
  \put(2,10.5){\line(1,0){18}}
  \put(14,4.5){\line(1,0){18}}
  \put(8,7.5){\line(1,0){30}}
  \put(26,10.5){\line(1,0){18}}
\end{picture}
$$
  \end{center}
\end{minipage}

\begin{minipage}{0.3\linewidth}
  \begin{center}
    \setlength{\unitlength}{3em}
    \begin{picture}(3.0,3.0)(-0.5,0.2)
      \put(1,2){\circle*{0.1}}
      \put(1.1,2.1){$1$}
      \put(0,1){\circle*{0.1}}
      \put(-0.2,1.1){$2$}
      \put(2,1){\circle*{0.1}}
      \put(2.1,1.1){$3$}
      \put(1,0){\circle*{0.1}}
      \put(1.2,-0.2){$4$}
      \put(1,2){\vector(1,-1){0.95}}
      \put(1,2){\vector(-1,-1){0.95}}
      \put(0,1){\vector(1,-1){0.95}}
      \put(2,1){\vector(-1,-1){0.95}}
    \end{picture}

$$
\setlength{\unitlength}{\standardunit}
\begin{picture}(44,9.5)(1,0)
  \put(2,0){\line(0,1){13.5}}
  \put(8,0){\line(0,1){10.5}}
  \put(14,0){\line(0,1){7.5}}
  \put(20,0){\line(0,1){13.5}}
  \put(26,0){\line(0,1){4.5}}
  \put(32,0){\line(0,1){7.5}}
  \put(38,0){\line(0,1){10.5}}
  \put(44,0){\line(0,1){4.5}}
  \put(2,13.5){\line(1,0){18}}
  \put(14,7.5){\line(1,0){18}}
  \put(8,10.5){\line(1,0){30}}
  \put(26,4.5){\line(1,0){18}}
\end{picture}
$$
  \end{center}
\end{minipage}
\begin{minipage}{0.3\linewidth}
  \begin{center}
    \setlength{\unitlength}{3em}
    \begin{picture}(3.0,3.0)(-0.5,0.2)
      \put(1,2){\circle*{0.1}}
      \put(1.1,2.1){$1$}
      \put(0,1){\circle*{0.1}}
      \put(-0.2,1.1){$2$}
      \put(2,1){\circle*{0.1}}
      \put(2.1,1.1){$3$}
      \put(1,0){\circle*{0.1}}
      \put(1.2,-0.2){$4$}
      \put(1,2){\vector(1,-1){0.95}}
      \put(1,2){\vector(-1,-1){0.95}}
      \put(0,1){\vector(1,-1){0.95}}
      \put(1,0){\vector(1,1){0.95}}
    \end{picture}

$$
\setlength{\unitlength}{\standardunit}
\begin{picture}(44,9.5)(1,0)
  \put(2,0){\line(0,1){13.5}}
  \put(8,0){\line(0,1){10.5}}
  \put(14,0){\line(0,1){4.5}}
  \put(20,0){\line(0,1){13.5}}
  \put(26,0){\line(0,1){7.5}}
  \put(32,0){\line(0,1){4.5}}
  \put(38,0){\line(0,1){10.5}}
  \put(44,0){\line(0,1){7.5}}
  \put(2,13.5){\line(1,0){18}}
  \put(14,4.5){\line(1,0){18}}
  \put(8,10.5){\line(1,0){30}}
  \put(26,7.5){\line(1,0){18}}
\end{picture}
$$
  \end{center}
\end{minipage}
\begin{minipage}{0.3\linewidth}
  \begin{center}
    \setlength{\unitlength}{3em}
    \begin{picture}(3.0,3.0)(-0.5,0.2)
      \put(1,2){\circle*{0.1}}
      \put(1.1,2.1){$1$}
      \put(0,1){\circle*{0.1}}
      \put(-0.2,1.1){$2$}
      \put(2,1){\circle*{0.1}}
      \put(2.1,1.1){$3$}
      \put(1,0){\circle*{0.1}}
      \put(1.2,-0.2){$4$}
      \put(1,2){\vector(1,-1){0.95}}
      \put(1,2){\vector(-1,-1){0.95}}
      \put(1,0){\vector(-1,1){0.95}}
      \put(2,1){\vector(-1,-1){0.95}}
    \end{picture}

$$
\setlength{\unitlength}{\standardunit}
\begin{picture}(44,9.5)(1,0)
  \put(2,0){\line(0,1){13.5}}
  \put(8,0){\line(0,1){4.5}}
  \put(14,0){\line(0,1){10.5}}
  \put(20,0){\line(0,1){13.5}}
  \put(26,0){\line(0,1){7.5}}
  \put(32,0){\line(0,1){10.5}}
  \put(38,0){\line(0,1){4.5}}
  \put(44,0){\line(0,1){7.5}}
  \put(2,13.5){\line(1,0){18}}
  \put(14,10.5){\line(1,0){18}}
  \put(8,4.5){\line(1,0){30}}
  \put(26,7.5){\line(1,0){18}}
\end{picture}
$$
  \end{center}
\end{minipage}

  \caption[Crossing heaps]{The crossing graph, some acyclic orientations and crossing heaps
    of the partition
$\pi=\mbox{\protect\begin{picture}(44,12.5)(1,0)
  \put(2,0){\line(0,1){4.5}}
  \put(8,0){\line(0,1){10.5}}
  \put(14,0){\line(0,1){7.5}}
  \put(20,0){\line(0,1){4.5}}
  \put(26,0){\line(0,1){4.5}}
  \put(32,0){\line(0,1){7.5}}
  \put(38,0){\line(0,1){10.5}}
  \put(44,0){\line(0,1){4.5}}
  \put(2,4.5){\line(1,0){18}}
  \put(14,7.5){\line(1,0){18}}
  \put(8,10.5){\line(1,0){30}}
  \put(26,4.5){\line(1,0){18}}
\protect\end{picture}}
$; blocks are numbered in the canonical order.
Note that the upper right orientation is not a pyramid.}
  \label{fig:crossingheaps}
\end{figure}

\begin{figure}
\begin{minipage}{0.3\linewidth}
  \begin{center}
    \setlength{\unitlength}{3em}
    \begin{picture}(4.0,4.0)(-0.5,0.2)
      \put(1,2){\circle*{0.1}}
      \put(1.1,2.1){$1$}
      \put(0,1){\circle*{0.1}}
      \put(-0.2,1.1){$2$}
      \put(2,1){\circle*{0.1}}
      \put(2.1,1.1){$4$}
      \put(1,0){\circle*{0.1}}
      \put(1.2,-0.2){$3$}
      \put(3,0){\circle*{0.1}}
      \put(3.2,-0.2){$5$}
      \put(1,2){\line(1,-1){1}}
      \put(1,2){\line(0,-1){2}}
      \put(1,2){\line(-1,-1){1}}
      \put(1,0){\line(-1,1){1}}
      \put(1,0){\line(1,1){1}}
      \put(2,1){\line(1,-1){1}}
    \end{picture}

$$
\setlength{\unitlength}{\standardunit}
\tilde{G}(
\begin{picture}(62,12.5)(1,0)
\put(2,0){\line(0,1){10.5}}
\put(8,0){\line(0,1){4.5}}
\put(14,0){\line(0,1){7.5}}
\put(20,0){\line(0,1){10.5}}
\put(26,0){\line(0,1){4.5}}
\put(32,0){\line(0,1){13.5}}
\put(38,0){\line(0,1){7.5}}
\put(44,0){\line(0,1){10.5}}
\put(50,0){\line(0,1){4.5}}
\put(56,0){\line(0,1){4.5}}
\put(62,0){\line(0,1){13.5}}
\put(8,4.5){\line(1,0){18}}
\put(14,7.5){\line(1,0){24}}
\put(2,10.5){\line(1,0){42}}
\put(50,4.5){\line(1,0){6}}
\put(32,13.5){\line(1,0){30}}
\end{picture})
$$
  \end{center}
\end{minipage}
\begin{minipage}{0.3\linewidth}
  \begin{center}
    \setlength{\unitlength}{3em}
    \begin{picture}(4.0,4.0)(-0.5,0.2)
      \put(1,2){\circle*{0.1}}
      \put(1.1,2.1){$1$}
      \put(0,1){\circle*{0.1}}
      \put(-0.2,1.1){$2$}
      \put(2,1){\circle*{0.1}}
      \put(2.1,1.1){$4$}
      \put(1,0){\circle*{0.1}}
      \put(1.2,-0.2){$3$}
      \put(3,0){\circle*{0.1}}
      \put(3.2,-0.2){$5$}
      \put(1,2){\vector(1,-1){0.95}}
      \put(1,2){\vector(0,-1){1.95}}
      \put(1,2){\vector(-1,-1){0.95}}
      \put(0,1){\vector(1,-1){0.95}}
      \put(2,1){\vector(-1,-1){0.95}}
      \put(2,1){\vector(1,-1){0.95}}
    \end{picture}

$$
\setlength{\unitlength}{\standardunit}
\begin{picture}(62,12.5)(1,0)
\put(2,0){\line(0,1){10.5}}
\put(8,0){\line(0,1){7.5}}
\put(14,0){\line(0,1){4.5}}
\put(20,0){\line(0,1){10.5}}
\put(26,0){\line(0,1){7.5}}
\put(32,0){\line(0,1){7.5}}
\put(38,0){\line(0,1){4.5}}
\put(44,0){\line(0,1){10.5}}
\put(50,0){\line(0,1){4.5}}
\put(56,0){\line(0,1){4.5}}
\put(62,0){\line(0,1){7.5}}
\put(8,7.5){\line(1,0){18}}
\put(14,4.5){\line(1,0){24}}
\put(2,10.5){\line(1,0){42}}
\put(50,4.5){\line(1,0){6}}
\put(32,7.5){\line(1,0){30}}
\end{picture}
$$
  \end{center}
\end{minipage}
\begin{minipage}{0.3\linewidth}
  \begin{center}
    \setlength{\unitlength}{3em}
    \begin{picture}(4.0,4.0)(-0.5,0.2)
      \put(1,2){\circle*{0.1}}
      \put(1.1,2.1){$1$}
      \put(0,1){\circle*{0.1}}
      \put(-0.2,1.1){$2$}
      \put(2,1){\circle*{0.1}}
      \put(2.1,1.1){$4$}
      \put(1,0){\circle*{0.1}}
      \put(1.2,-0.2){$3$}
      \put(3,0){\circle*{0.1}}
      \put(3.2,-0.2){$5$}
      \put(1,2){\vector(1,-1){0.95}}
      \put(1,2){\vector(0,-1){1.95}}
      \put(1,2){\vector(-1,-1){0.95}}
      \put(1,0){\vector(-1,1){0.95}}
      \put(2,1){\vector(-1,-1){0.95}}
      \put(2,1){\vector(1,-1){0.95}}
    \end{picture}

$$
\setlength{\unitlength}{\standardunit}
\begin{picture}(62,12.5)(1,0)
\put(2,0){\line(0,1){13.5}}
\put(8,0){\line(0,1){4.5}}
\put(14,0){\line(0,1){7.5}}
\put(20,0){\line(0,1){13.5}}
\put(26,0){\line(0,1){4.5}}
\put(32,0){\line(0,1){10.5}}
\put(38,0){\line(0,1){7.5}}
\put(44,0){\line(0,1){13.5}}
\put(50,0){\line(0,1){4.5}}
\put(56,0){\line(0,1){4.5}}
\put(62,0){\line(0,1){10.5}}
\put(8,4.5){\line(1,0){18}}
\put(14,7.5){\line(1,0){24}}
\put(2,13.5){\line(1,0){42}}
\put(50,4.5){\line(1,0){6}}
\put(32,10.5){\line(1,0){30}}
\end{picture}
$$
  \end{center}
\end{minipage}
\begin{minipage}{0.3\linewidth}
  \begin{center}
    \setlength{\unitlength}{3em}
    \begin{picture}(4.0,4.0)(-0.5,0.2)
      \put(1,2){\circle*{0.1}}
      \put(1.1,2.1){$1$}
      \put(0,1){\circle*{0.1}}
      \put(-0.2,1.1){$2$}
      \put(2,1){\circle*{0.1}}
      \put(2.1,1.1){$4$}
      \put(1,0){\circle*{0.1}}
      \put(1.2,-0.2){$3$}
      \put(3,0){\circle*{0.1}}
      \put(3.2,-0.2){$5$}
      \put(1,2){\vector(1,-1){0.95}}
      \put(1,2){\vector(0,-1){1.95}}
      \put(1,2){\vector(-1,-1){0.95}}
      \put(0,1){\vector(1,-1){0.95}}
      \put(1,0){\vector(1,1){0.95}}
      \put(2,1){\vector(1,-1){0.95}}
    \end{picture}

$$
\setlength{\unitlength}{\standardunit}
\begin{picture}(62,14.5)(1,0)
\put(2,0){\line(0,1){16.5}}
\put(8,0){\line(0,1){13.5}}
\put(14,0){\line(0,1){10.5}}
\put(20,0){\line(0,1){16.5}}
\put(26,0){\line(0,1){13.5}}
\put(32,0){\line(0,1){7.5}}
\put(38,0){\line(0,1){10.5}}
\put(44,0){\line(0,1){16.5}}
\put(50,0){\line(0,1){4.5}}
\put(56,0){\line(0,1){4.5}}
\put(62,0){\line(0,1){7.5}}
\put(8,13.5){\line(1,0){18}}
\put(14,10.5){\line(1,0){24}}
\put(2,16.5){\line(1,0){42}}
\put(50,4.5){\line(1,0){6}}
\put(32,7.5){\line(1,0){30}}
\end{picture}
$$
  \end{center}
\end{minipage}
\begin{minipage}{0.3\linewidth}
  \begin{center}
    \setlength{\unitlength}{3em}
    \begin{picture}(4.0,4.0)(-0.5,0.2)
      \put(1,2){\circle*{0.1}}
      \put(1.1,2.1){$1$}
      \put(0,1){\circle*{0.1}}
      \put(-0.2,1.1){$2$}
      \put(2,1){\circle*{0.1}}
      \put(2.1,1.1){$4$}
      \put(1,0){\circle*{0.1}}
      \put(1.2,-0.2){$3$}
      \put(3,0){\circle*{0.1}}
      \put(3.2,-0.2){$5$}
      \put(1,2){\vector(1,-1){0.95}}
      \put(1,2){\vector(0,-1){1.95}}
      \put(1,2){\vector(-1,-1){0.95}}
      \put(1,0){\vector(-1,1){0.95}}
      \put(1,0){\vector(1,1){0.95}}
      \put(2,1){\vector(1,-1){0.95}}
    \end{picture}

$$
\setlength{\unitlength}{\standardunit}
\begin{picture}(62,14.5)(1,0)
\put(2,0){\line(0,1){13.5}}
\put(8,0){\line(0,1){7.5}}
\put(14,0){\line(0,1){10.5}}
\put(20,0){\line(0,1){13.5}}
\put(26,0){\line(0,1){7.5}}
\put(32,0){\line(0,1){7.5}}
\put(38,0){\line(0,1){10.5}}
\put(44,0){\line(0,1){13.5}}
\put(50,0){\line(0,1){4.5}}
\put(56,0){\line(0,1){4.5}}
\put(62,0){\line(0,1){7.5}}
\put(8,7.5){\line(1,0){18}}
\put(14,10.5){\line(1,0){24}}
\put(2,13.5){\line(1,0){42}}
\put(50,4.5){\line(1,0){6}}
\put(32,7.5){\line(1,0){30}}
\end{picture}
$$
  \end{center}
\end{minipage}
\begin{minipage}{0.3\linewidth}
  \begin{center}
    \setlength{\unitlength}{3em}
    \begin{picture}(4.0,4.0)(-0.5,0.2)
      \put(1,2){\circle*{0.1}}
      \put(1.1,2.1){$1$}
      \put(0,1){\circle*{0.1}}
      \put(-0.2,1.1){$2$}
      \put(2,1){\circle*{0.1}}
      \put(2.1,1.1){$4$}
      \put(1,0){\circle*{0.1}}
      \put(1.2,-0.2){$3$}
      \put(3,0){\circle*{0.1}}
      \put(3.2,-0.2){$5$}
      \put(1,2){\vector(1,-1){0.95}}
      \put(1,2){\vector(0,-1){1.95}}
      \put(1,2){\vector(-1,-1){0.95}}
      \put(1,0){\vector(-1,1){0.95}}
      \put(1,0){\vector(1,1){0.95}}
      \put(3,0){\vector(-1,1){0.95}}
    \end{picture}

$$
\setlength{\unitlength}{\standardunit}
\begin{picture}(62,14.5)(1,0)
\put(2,0){\line(0,1){13.5}}
\put(8,0){\line(0,1){7.5}}
\put(14,0){\line(0,1){10.5}}
\put(20,0){\line(0,1){13.5}}
\put(26,0){\line(0,1){7.5}}
\put(32,0){\line(0,1){7.5}}
\put(38,0){\line(0,1){10.5}}
\put(44,0){\line(0,1){13.5}}
\put(50,0){\line(0,1){10.5}}
\put(56,0){\line(0,1){10.5}}
\put(62,0){\line(0,1){7.5}}
\put(8,7.5){\line(1,0){18}}
\put(14,10.5){\line(1,0){24}}
\put(2,13.5){\line(1,0){42}}
\put(50,10.5){\line(1,0){6}}
\put(32,7.5){\line(1,0){30}}
\end{picture}
$$
  \end{center}
\end{minipage}

  \caption[Interval heaps]{The anti-interval graph, some acyclic orientations
    and interval heaps
    of the partition
$\pi=
\mbox{
\protect\begin{picture}(62,14.5)(1,0)
\put(2,0){\line(0,1){10.5}}
\put(8,0){\line(0,1){4.5}}
\put(14,0){\line(0,1){7.5}}
\put(20,0){\line(0,1){10.5}}
\put(26,0){\line(0,1){4.5}}
\put(32,0){\line(0,1){13.5}}
\put(38,0){\line(0,1){7.5}}
\put(44,0){\line(0,1){10.5}}
\put(50,0){\line(0,1){4.5}}
\put(56,0){\line(0,1){4.5}}
\put(62,0){\line(0,1){13.5}}
\put(8,4.5){\line(1,0){18}}
\put(14,7.5){\line(1,0){24}}
\put(2,10.5){\line(1,0){42}}
\put(50,4.5){\line(1,0){6}}
\put(32,13.5){\line(1,0){30}}
\protect\end{picture}}
$. All heaps except  the last are pyramidal.
}

  \label{fig:intervalheaps}
\end{figure}



\begin{prop}
  \label{prop:crossingheap}
  \begin{enumerate}
   \item  The crossing heaps of a connected partition are in bijection with the
    acyclic orientations of its crossing graph. Pyramids correspond to those rooted
    acyclic orientations which are rooted in the first block,
    thus $T_{G(\pi)}(1,0)$ equals the number of crossing heaps on $\pi$ which
    are pyramids.
   \item The interval heaps of an irreducible partition are in bijection with the
    acyclic orientations of its anti-interval graph. Pyramids correspond to
    those rooted
    acyclic orientations which are rooted in the first block,
    thus $T_{\tilde{G}(\pi)}(1,0)$ equals the number of interval heaps on $\pi$ which
    are pyramids.
  \end{enumerate}
\end{prop}

From these facts one can make a connection to
the Cartier-Foata-Viennot theory of heaps \cite{CartierFoata:1969:problemes}
and in fact some proofs of \cite{JV13} make use of this machinery. 
The heap  interpretation of formulas \eqref{eq:free2class} and
\eqref{eq:bool2class} read as follows.
\begin{cor}
  \label{cor:pyramidformula}
  The classical cumulants can be expressed in terms of the free and Boolean
  cumulants as follows.
  \begin{equation}
    \label{eq:crossingpyramids}
    K_n = \sum (-1)^{\abs{\pi}-1} R_\pi
  \end{equation}
    where the sum runs over all pyramidal crossing heaps.
    \begin{equation}
  \label{eq:intervalpyramids}
    K_n = \sum (-1)^{\abs{\pi}-1} B_\pi
    \end{equation}
    where the sum runs over all pyramidal interval heaps.
\end{cor}

\section{Permutation statistics and proof of Theorem~\ref{thm:bool2classcycles}}\label{permutation}
There are no cancellations in formula \eqref{eq:bool2class} and one might
wonder whether there is another combinatorial interpretation.
The following formulas suggest that a statistic on permutations might be
involved. 
\begin{cor}\label{cor9}
  $$  
  \sum_{\pi\in\irr\SP(n)} T_{\tilde{G}(\pi)}(1,0) = (n-1)!. 
  $$
\end{cor}

\begin{proof}
  The sum is obtained by evaluating the negative classical cumulants
  of a (formal) distribution with Boolean cumulants
  $b_n=-1$ for all $n$. The ``moments'' have generating function
  $$
  M(z) = \frac{1}{1+\sum_{n=1}^\infty z^n} = 1-z
  $$
  and in this case $F(z)=M(z)$.
  Thus the negative classical cumulants have exponential generating series
  $$
  - \log (1-z) = \sum_{n=1}^\infty \frac{z^n}{n}. 
  $$
\end{proof}

We collect key concepts for statistics of permutations. 
\begin{enumerate}
 \item 
Recall that a \emph{run} in a permutation $\sigma \in\SG_n$ is a maximal
increasing segment of the sequence $(\sigma(1),
\sigma(2),\dots,\sigma(n))$. Any permutation decomposes uniquely into runs.
With the exception of the last run, at the end of each run there is a descent,
therefore the number of runs is equal to the number of descents incremented by
one. Given a permutation $\sigma$ we denote by $\runs(\sigma)$ the set
partition associated to the set of runs. 

\item 
Every permutation $\pi$ has a unique factorization into cycles
$\sigma=\gamma_1\gamma_2\dotsm\gamma_k$,
where  the cycles are sorted in increasing order
with respect to their minimal elements.
Moreover, every cycle is written starting with its minimal element.
A \emph{cycle run} in a permutation $\sigma$ is a maximal contiguous
increasing subsequence of one of its cycles;
in other words, a maximal sequence $i_1<i_2<\dotsm< i_r$
such that $\sigma(i_k)=i_{k+1}$.

It is easy to see that distinct cycle runs of a given permutation are disjoint
and therefore give rise a set partition of order $n$. 
We denote this set partition by $\cycleruns(\sigma)$. 

We denote by $\cycles(\sigma)\in\SP(n)$
the \emph{cycle partition} of $\sigma$, i.e., the set partition whose blocks are 
given by the cycles of $\sigma$.
\end{enumerate}

The key identity for showing Theorem \ref{thm:bool2classcycles} is contained in
the  following lemma. 
\begin{lem}\label{thm7} 
\begin{equation}  \label{eq:bool2classpermutation}
    \varphi(X_1X_2\dotsm X_n) = \sum_{\sigma\in \SG_n} (-1)^{\#\cycleruns(\sigma)-\#\cycles(\sigma)} B_{\cycleruns(\sigma)}(X_1,X_2,\dots,X_n). 
  \end{equation}
\end{lem}
\begin{proof} The starting point is formula \eqref{mcboolean} for $\pi=\hat{1}_n$, expressing moments in terms of Boolean cumulants. 
 The right hand side of formula \eqref{mcboolean} for $\pi=\hat{1}_n$ arise
 from the terms $$
 (-1)^{\#\cycleruns(\sigma)-\#\cycles(\sigma)}
 B_{\cycleruns(\sigma)}(X_1,X_2,\dots,X_n)
 $$ 
 where $\sigma$ factorizes into
 ``interval cycles'', i.e., cycles of the form    
 $$
 (k,k+1,k+2,\dots,l). 
 $$
 Note that in this case we have $\#\cycles(\sigma)=\#\cycleruns(\sigma)$ and so the sign is $+1$. 
 Hence the proof follows the common strategy of finding an inversion which shows cancellation of all permutations which do not have a factorization into interval cycles. 

  Assume that a permutation $\sigma$ is not of interval type,
  and let $\sigma=\gamma_1\gamma_2\dotsm \gamma_m$ be its
  cycle decomposition in standard order and
  $\gamma_i=(r_{i1},r_{i2},\dots, r_{i k_i})$ be the run decompositions
  of the cycles.
  Since $\sigma$ is not of interval type, there are descents.
  Our involution $\Phi$ operates on the last descent.
  We say that $\sigma$ is  of type $A$ if the last descent occurs
  within one cycle and type $B$ if it occurs between two consecutive cycles.
  We set up an involution $\Phi$ as follows.  

  If $\sigma$ is of type $A$ and the last descent occurs within one cycle,
  then it occurs necessarily
  between the last two runs of this cycle. We split the cycle at this
  descent to obtain $\Phi(\sigma)$ which is of type $B$.

  If $\sigma$ is of type $B$, the descent occurs between two cycles and the
  second cycle necessarily only consists of one run. We join the two cycles
  to obtain  $\Phi(\sigma)$ which is of type $A$.

  Figure~\ref{fig:involution} shows an example of the action of 
  the involution $\Phi$. The permutation on the left (type A)
  is mapped to the permutation on the right (type B) and vice versa.

  \begin{figure}
    \centering
    \begin{minipage}{0.3\linewidth}
    \setlength{\unitlength}{1em}
    \begin{picture}(9.0,9.0)
      \put(1,1){\line(1,2){1.0}}
      \put(1,1){\circle*{0.2}}
      \put(2,3){\circle*{0.2}}
      \put(3,2){\line(1,3){1.0}}
      \put(3,2){\circle*{0.2}}
      \put(4,5){\line(1,2){1.0}}
      \put(4,5){\circle*{0.2}}
      \put(5,7){\line(1,-3){1.0}}
      \put(5,7){\circle*{0.2}}
      \put(6,4){\line(1,2){1.0}}
      \put(6,4){\circle*{0.2}}
      \put(7,6){\circle*{0.2}}
      \put(8,8){\line(1,1){1.0}}
      \put(8,8){\circle*{0.2}}
      \put(9,9){\circle*{0.2}}
    \end{picture}
    \phantom{x}$(1,3)(2,5,7,4,6)(8,9)$
    \end{minipage}
    \begin{minipage}{0.3\linewidth}
    \setlength{\unitlength}{1em}
    \begin{picture}(9.0,9.0)
      \put(1,1){\line(1,2){1.0}}
      \put(1,1){\circle*{0.2}}
      \put(2,3){\circle*{0.2}}
      \put(3,2){\line(1,3){1.0}}
      \put(3,2){\circle*{0.2}}
      \put(4,5){\line(1,2){1.0}}
      \put(4,5){\circle*{0.2}}
      \multiput(5,7)(0.2,-0.6){5}{\line(1,-3){0.1}}
      \put(5,7){\circle*{0.2}}
      \put(6,4){\line(1,2){1.0}}
      \put(6,4){\circle*{0.2}}
      \put(7,6){\circle*{0.2}}
      \put(8,8){\line(1,1){1.0}}
      \put(8,8){\circle*{0.2}}
      \put(9,9){\circle*{0.2}}
    \end{picture}
    \phantom{x}$(1,3)(2,5,7)(4,6)(8,9)$
    \end{minipage}
    \caption{The involution $\Phi$}
    \label{fig:involution}
  \end{figure}
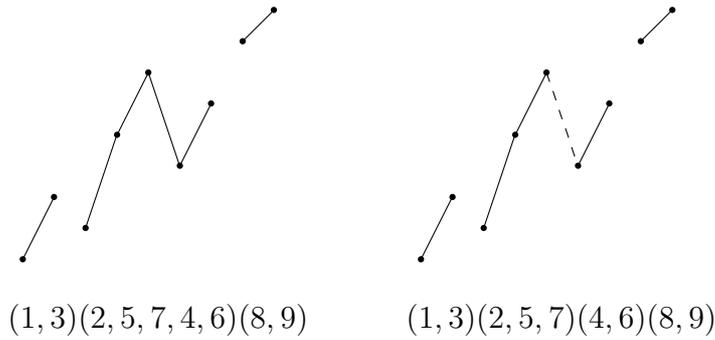

  In both cases
  the position of the last descent remains the same
  and therefore $\Phi$ maps type $A$ to type $B$ bijectively
  and $\Phi\circ\Phi=\text{Id}$.

  On the other hand, the total number of runs is unchanged,
  while the total number of cycles is changed by $\pm1$.
  It follows that the contributions of type A and type B permutations
  in the sum~\eqref{eq:bool2classpermutation} cancel.
\end{proof}

\begin{proof}[First proof of Theorem~\ref{thm:bool2classcycles}]
  Denote by $\tilde{K}_n$ the right hand side of (\ref{eq:bool2classcycles}): 
  $$
\tilde{K}_n=\sum_{\sigma\in \Cycles_n} (-1)^{\#\cycleruns(\sigma)-1} B_{\cycleruns(\sigma)}. 
  $$
  The full cycles are exactly the permutations $\sigma$
  such that $\cycles(\sigma)=\hat{1}_n$ and thus
  $$
  \tilde{K}_n=\sum_{\substack{\sigma\in\SG_n\\ \cycles(\sigma)=\hat{1}_n}}(-1)^{\#\cycleruns(\sigma)-1} B_{\cycleruns(\sigma)}.
  $$
Taking the product of them, it is then easy to see that
  $\tilde{K}_\pi=\sum_{\substack{\sigma\in\SG_n\\ \cycles(\sigma)=\pi}}(-1)^{\#\cycleruns(\sigma)-\#\cycles(\sigma)} B_{\cycleruns(\sigma)},$
  and so from Lemma \ref{thm7}
  $$
  \varphi(X_1\cdots X_n) = \sum_{\pi\in\SP(n)} \tilde{K}_\pi(X_1,\dots,X_n). 
  $$
  Since the same formula holds with $\tilde{K}_\pi$ replaced by $K_\pi$, 
  then multiplicativity and the M\"obius principle imply that $\tilde{K}_n=K_n$.
\end{proof}

Another proof of Theorem~\ref{thm:bool2classcycles} can be obtained from the
following identities.
\begin{lem}
  \label{lem:intervaltutte}
  For any irreducible partition $\pi\in \SP(n)$ the evaluation of the Tutte polynomial
  $T_{\tilde{G}(\pi)}(1,0)$ is equal to the following numbers. 
  \begin{enumerate}
   \item The number of pyramidal interval heaps on $\pi$.
   \item The number of cyclic permutations $\sigma\in\Cycles_n$ such that 
    and $\cycleruns(\sigma)=\pi$.
   \item The number of permutations $\sigma\in\SG_n$ such that $\sigma(1)=1$
    and $\runs(\sigma)=\pi$.
  \end{enumerate}
  In particular, $\cycleruns(\sigma)$ is an irreducible permutation for every $\sigma\in\Cycles_n$.
\end{lem}

\begin{proof}
  The first identity has already been seen in
  Proposition~\ref{prop:crossingheap} and the equality of the other two 
  numbers
  follows from the obvious bijection between cyclic permutations and
  permutations fixing $1$.
  It remains to provide a bijection between runs of cyclic permutations and
  pyramidal interval heaps.
  
  Given a cyclic permutation~$\sigma$, we assign to it 
  an ordered interval partition $\Psi(\sigma)$
  consisting of the cycle runs of $\sigma$ read from left to right.
  This order induces a heap structure by placing subsequent blocks
  below their predecessors. We claim that $\Psi(\sigma)$ is a pyramid.

  Let $r_1$, $r_2$, \ldots{}, $r_m$ be the cycle runs of $\sigma$,
  in the order of their appearance when $\sigma$ is written as 
  $(1,\sigma(1),\sigma^2(1),\dots,\sigma^{n-1}(1))$.
  Clearly $r_1$ is not a singleton and for every run $r_k$ with $k\geq 2$ the following statements are true. 
  \begin{enumerate}[(a)]
   \item The maximal element in the run $r_{l-1}$ exceeds the starting element of $r_l$ $(2 \leq l \leq k)$. 
   
   \item There is a run $r_i$ with $i\leq k-1$ such that
    as an element of the heap, $r_i$ covers $r_k$, i.e.,
    $\max r_i  > \min r_k$ and $\min r_i<\max r_k$. In other words,
    the convex hulls of $r_i$ and $r_k$ intersect.
 \end{enumerate}
Since $r_1,\dots, r_{l-1},r_l, \dots$ is the decomposition of $\sigma$ into cycle runs, we must have $\max r_{l-1} > \min r_l$ and this implies (a).

 To show (b), we perform the following steps.
  \begin{enumerate}
   \item If $\min r_{k-1} < \max r_k$, then $r_{k-1}$ itself covers $r_k$ from (a) and we are done.
   \item If $\min r_{k-1} > \max r_k$, then $\max r_{k-2} > \min r_{k-1} > \max r_k$. If $\min r_{k-2} < \max r_k$, then $r_{k-2}$ covers $r_k$ and we are done. If not, we go to $r_{k-3}$. We repeat this until we reach some run $r_{i}$ $(i < k-1)$ such that $\min r_{i} < \max
    r_k < \max r_i$.
    This must happen at some point because ultimately we reach 
    $r_1=(1,\dots)$ and then clearly $1=\min r_1 < \max r_k$.
      \end{enumerate}

  The inverse map can be defined recursively as follows:
  Given an interval heap, take the leftmost minimal element and write it to the
  left of the previously written cycle run.
  There are two possibilities.
  Either the previous run was not covered by the current one,
  then it came from a block to the left,
  or it was covered by the current one.
  In either case the maximum of the current block is larger than the
  minimum of the previous one and thus the current block starts a new cycle run.
  For an example of this process see Figure~\ref{fig:inversemap}.
  This map clearly reverses the map $\Psi$ defined above.
  \begin{figure}


\hfill{}
\begin{minipage}{0.4\linewidth}
$$
\setlength{\unitlength}{0.2em}
\begin{aligned}
&\begin{picture}(74,15.5)(1,3)
\put(2,14.5){\line(0,1){2}}
{\color{red}{\put(8,2.5){\line(0,1){2}}}}
\put(14,5.5){\line(0,1){2}}
\put(20,11.5){\line(0,1){2}}
{\color{red}{\put(26,2.5){\line(0,1){2}}}}
\put(32,14.5){\line(0,1){2}}
\put(38,8.5){\line(0,1){2}}
\put(44,5.5){\line(0,1){2}}
\put(50,2.5){\line(0,1){2}}
\put(56,11.5){\line(0,1){2}}
\put(62,2.5){\line(0,1){2}}
\put(68,14.5){\line(0,1){2}}
\put(74,8.5){\line(0,1){2}}
{\color{red}{\put(8,4.5){\line(1,0){18}}}}
\put(14,7.5){\line(1,0){30}}
\put(20,13.5){\line(1,0){36}}
\put(50,4.5){\line(1,0){12}}
\put(2,16.5){\line(1,0){66}}
\put(38,10.5){\line(1,0){36}}
\end{picture}\\
&
\newcommand{\dista}{\,\,\,\,}
1\dista{}2\dista{}3\dista{}4\dista{}5\dista{}6\dista{}7\dista{}8\dista{}9\dista{}10\,11\,12\,13
\end{aligned}
$$
\end{minipage}
\hfill{}
$\to$
\hfill{}
(\phantom{1,6,12,4,10,7,13,9,11,3,8,}{\color{red}2,5})
\hfill{}


\hfill{}
\begin{minipage}{0.4\linewidth}
$$
\setlength{\unitlength}{0.2em}
\begin{aligned}
&\begin{picture}(74,15.5)(1,6)
\put(2,14.5){\line(0,1){2}}
{\color{red}\put(14,5.5){\line(0,1){2}}}
\put(20,11.5){\line(0,1){2}}
\put(32,14.5){\line(0,1){2}}
\put(38,8.5){\line(0,1){2}}
{\color{red}\put(44,5.5){\line(0,1){2}}}
\put(50,5.5){\line(0,1){2}}
\put(56,11.5){\line(0,1){2}}
\put(62,5.5){\line(0,1){2}}
\put(68,14.5){\line(0,1){2}}
\put(74,8.5){\line(0,1){2}}
{\color{red}\put(14,7.5){\line(1,0){30}}}
\put(20,13.5){\line(1,0){36}}
\put(50,7.5){\line(1,0){12}}
\put(2,16.5){\line(1,0){66}}
\put(38,10.5){\line(1,0){36}}
\end{picture}\\
&
\newcommand{\dista}{\,\,\,\,}
1\dista{}2\dista{}3\dista{}4\dista{}5\dista{}6\dista{}7\dista{}8\dista{}9\dista{}10\,11\,12\,13
\end{aligned}
$$
\end{minipage}
\hfill{}
$\to$
\hfill{}
(\phantom{1,6,12,4,10,7,13,9,11,}{\color{red}3,8},2,5)
\hfill{}


\hfill{}
\begin{minipage}{0.4\linewidth}
$$
\setlength{\unitlength}{0.2em}
\begin{aligned}
&\begin{picture}(74,15.5)(1,6)
\put(2,14.5){\line(0,1){2}}
\put(20,11.5){\line(0,1){2}}
\put(32,14.5){\line(0,1){2}}
\put(38,8.5){\line(0,1){2}}
{\color{red}\put(50,5.5){\line(0,1){2}}}
\put(56,11.5){\line(0,1){2}}
{\color{red}\put(62,5.5){\line(0,1){2}}}
\put(68,14.5){\line(0,1){2}}
\put(74,8.5){\line(0,1){2}}
\put(20,13.5){\line(1,0){36}}
{\color{red}\put(50,7.5){\line(1,0){12}}}
\put(2,16.5){\line(1,0){66}}
\put(38,10.5){\line(1,0){36}}
\end{picture}\\
&
\newcommand{\dista}{\,\,\,\,}
1\dista{}2\dista{}3\dista{}4\dista{}5\dista{}6\dista{}7\dista{}8\dista{}9\dista{}10\,11\,12\,13
\end{aligned}
$$
\end{minipage}
\hfill{}
$\to$
\hfill{}
(\phantom{1,6,12,4,10,7,13,}{\color{red}9,11},3,8,2,5)
\hfill{}


\hfill{}
\begin{minipage}{0.4\linewidth}
$$
\setlength{\unitlength}{0.2em}
\begin{aligned}
&\begin{picture}(74,15.5)(1,9)
\put(2,14.5){\line(0,1){2}}
\put(20,11.5){\line(0,1){2}}
\put(32,14.5){\line(0,1){2}}
{\color{red}\put(38,8.5){\line(0,1){2}}}
\put(56,11.5){\line(0,1){2}}
\put(68,14.5){\line(0,1){2}}
{\color{red}\put(74,8.5){\line(0,1){2}}}
\put(20,13.5){\line(1,0){36}}
\put(2,16.5){\line(1,0){66}}
{\color{red}\put(38,10.5){\line(1,0){36}}}
\end{picture}\\
&
\newcommand{\dista}{\,\,\,\,}
1\dista{}2\dista{}3\dista{}4\dista{}5\dista{}6\dista{}7\dista{}8\dista{}9\dista{}10\,11\,12\,13
\end{aligned}
$$
\end{minipage}
\hfill{}
$\to$
\hfill{}
(\phantom{1,6,12,4,10,}{\color{red}7,13},9,11,3,8,2,5)
\hfill{}


\hfill{}
\begin{minipage}{0.4\linewidth}
$$
\setlength{\unitlength}{0.2em}
\begin{aligned}
&\begin{picture}(74,15.5)(1,12)
\put(2,14.5){\line(0,1){2}}
{\color{red}\put(20,11.5){\line(0,1){2}}}
\put(32,14.5){\line(0,1){2}}
{\color{red}\put(56,11.5){\line(0,1){2}}}
\put(68,14.5){\line(0,1){2}}
{\color{red}\put(20,13.5){\line(1,0){36}}}
\put(2,16.5){\line(1,0){66}}
\end{picture}\\
&
\newcommand{\dista}{\,\,\,\,}
1\dista{}2\dista{}3\dista{}4\dista{}5\dista{}6\dista{}7\dista{}8\dista{}9\dista{}10\,11\,12\,13
\end{aligned}
$$
\end{minipage}
\hfill{}
$\to$
\hfill{}
(\phantom{1,6,12,}{\color{red}4,10},7,13,9,11,3,8,2,5)
\hfill{}


\hfill{}
\begin{minipage}{0.4\linewidth}
$$
\setlength{\unitlength}{0.2em}
\begin{aligned}
&\begin{picture}(74,15.5)(1,15)
{\color{red}\put(2,14.5){\line(0,1){2}}
\put(32,14.5){\line(0,1){2}}
\put(68,14.5){\line(0,1){2}}
\put(2,16.5){\line(1,0){66}}
}
\end{picture}\\
&
\newcommand{\dista}{\,\,\,\,}
1\dista{}2\dista{}3\dista{}4\dista{}5\dista{}6\dista{}7\dista{}8\dista{}9\dista{}10\,11\,12\,13
\end{aligned}
$$
\end{minipage}
\hfill{}
$\to$
\hfill{}
({\color{red}1,6,12},4,10,7,13,9,11,3,8,2,5)
\hfill{}    
    \caption{The inverse map for an interval heap on the partition
$
\pi =
\left\{
{\left\{ 1, \: 6, \: {12} 
\right\}},
\: {\left\{ 2, \: 5 
\right\}},
\: {\left\{ 3, \: 8 
\right\}},
\: {\left\{ 4, \: {10} 
\right\}},
\: {\left\{ 7, \: {13} 
\right\}},
\: {\left\{ 9, \: {11} 
\right\}}
\right\}
$
}
    \label{fig:inversemap}
  \end{figure}
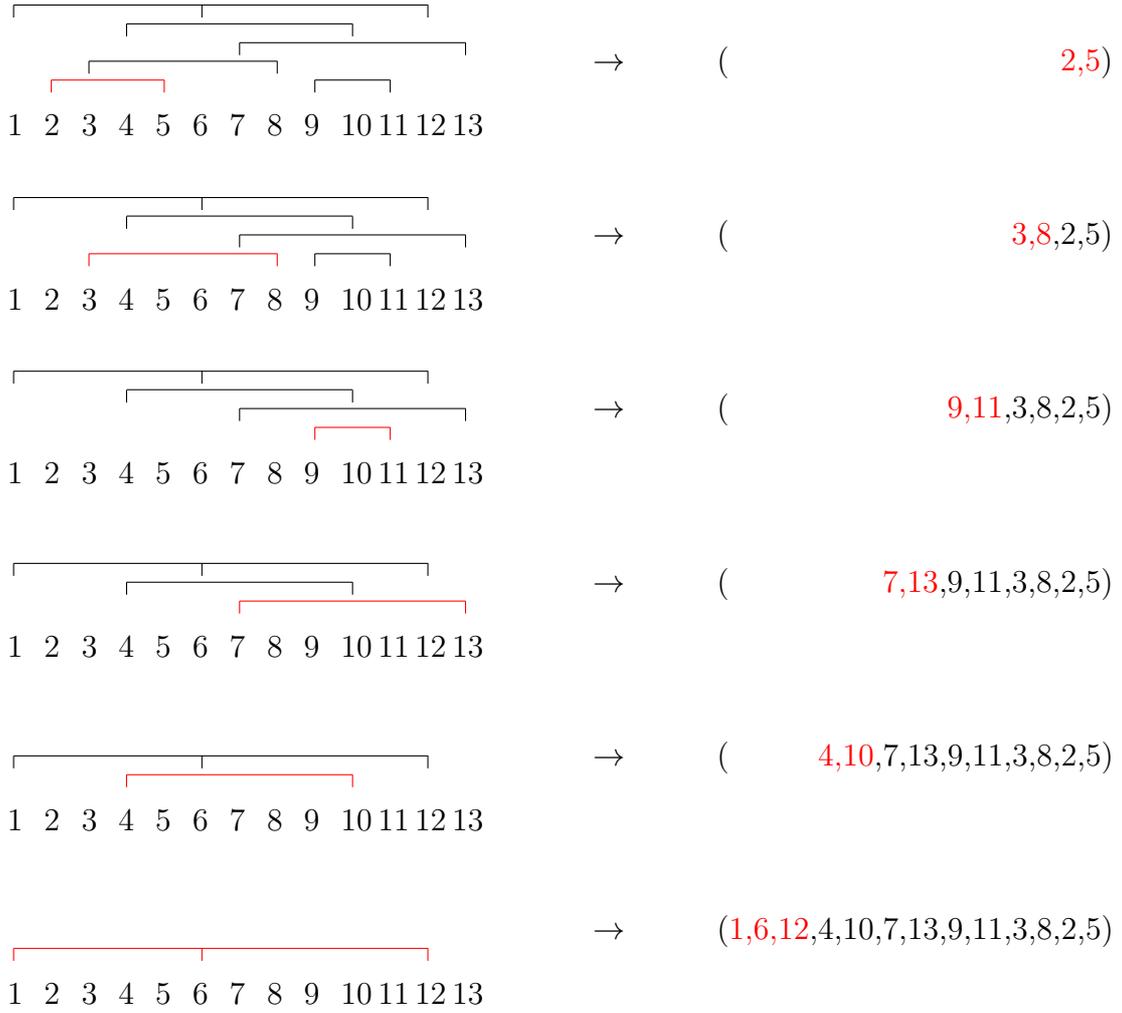
\end{proof}
With this lemma, Theorem~\ref{thm:bool2classcycles} follows from Corollary \ref{cor:pyramidformula}, equation \eqref{eq:intervalpyramids} or from Theorem \ref{mainthm3}. 

Recall that the \emph{Eulerian polynomials} are defined by
\begin{equation}\label{euler}
E_n(x) = \sum_{\sigma\in \SG_n} x^{d(\sigma)} = \sum_{k=0}^n \eulerbr{n}{k}\, x^k
\end{equation}
where $\eulerbr{n}{k}$ is the number of permutations with $k$ descents.
A good reference for these and the following facts is the book \cite{GrahamKnuthPatashnik:1994:concrete}.
The Eulerian polynomials have been studied as moments in
\cite{Barry:2011:eulerian}. 

Note that each run of a permutation must follow a descent and
therefore $\sharp(\runs(\sigma))=d(\sigma)+1$,
and Lemma~\ref{lem:intervaltutte} together with formula \eqref{euler} yields the identity 
\begin{equation}
\sum_{\substack{\pi \in \irr\SP(n) \\ \abs{\pi}=k}} T_{\tilde{G}(\pi)}(1,0) =  \eulerbr{n-1}{k-1}. 
\end{equation}

As a special case of Corollary \ref{thm:bool2classpermutation1}, we consider the Boolean Poisson distribution. We interpret the Eulerian polynomials as classical cumulants.

\begin{prop}\label{prop10}
  Consider the distribution with Boolean cumulants $b_n=x$ for all $n$.   Then the classical cumulants are given by the Eulerian polynomials
  $$
  \kappa_n = xE_{n-1}(-x). 
$$
\end{prop}
\begin{proof}
By the remarks above $b_{\runs(\sigma)}=x^{d(\sigma)+1}$, and hence from Theorem \ref{thm:bool2classpermutation1}
$$
\kappa_n=\sum_{\substack{\sigma\in \SG_n\\ \sigma(1)=1}} (-1)^{d(\sigma)} x^{d(\sigma)+1}. 
$$
The desired formula follows from the natural bijection $\{\sigma\in \SG_n: \sigma(1)=1\} \to \SG_{n-1}$, which preserves the number of descents. 
\end{proof}

\begin{rem}
  \begin{enumerate} 
   \item The following determinant formulas relating moments and classical cumulants are known:
    $$
    \kappa_n = (-1)^{n-1} (n-1)!
    \begin{vmatrix}
      m_1 & 1 & 0 & 0 & 0 &\dotsm &0\\
      m_2 &m_1& 1 & 0 & 0&\dotsm &0\\
      \frac{m_3}{2!} & \frac{m_2}{2!} & m_1 & 1 &  0&\dotsm&0 \\
      \frac{m_4}{3!} & \frac{m_3}{3!} & \frac{m_2}{2!} & m_1 & 1 &  \dotsm&0 \\
      \vdots&\vdots&\vdots&\vdots&&&\vdots\\
      \frac{m_n}{(n-1)!} & \frac{m_{n-1}}{(n-1)!} & \frac{m_{n-2}}{(n-2)!}
      & &\dotsm &  &    m_1
    \end{vmatrix}
    $$
    and 
    $$
    m_n = 
    \begin{vmatrix}
      \kappa_1 & -1 & 0 & 0 & 0 &\dotsm &0\\
      \kappa_2 &\kappa_1& -2 & 0 & 0&\dotsm &0\\
      \frac{\kappa_3}{2!} & \kappa_2 & \kappa_1 & -3 &  0&\dotsm&0 \\
      \frac{\kappa_4}{3!} & \frac{\kappa_3}{2!} & \kappa_2 & \kappa_1 & -4 &  \dotsm&0 \\
      \vdots&\vdots&\vdots&\vdots&&&\vdots\\
      \frac{\kappa_n}{(n-1)!} & \frac{\kappa_{n-1}}{(n-2)!} & \frac{\kappa_{n-2}}{(n-3)!}
      & &\dotsm &  &    \kappa_1
    \end{vmatrix}. 
    $$ 
    These formulas follow from Cramer's rule applied to the linear system
    obtained from the recursion formula
    $$
    \sum_{k=1}^n \binom{n-1}{k-1} m_{n-k}\kappa_k = m_n,\qquad n \geq 1; 
    $$
    see \cite{RotaShen:2000:combinatorics}.
   \item 
    Similarly, comparison of coefficients in the identity  $M(z)(1-B(z))=1$ 
    leads to the recursion
    $$
   \sum_{k=1}^{n}m_{n-k}b_{k} = m_n, \qquad n \geq 1
    $$
    and thus
    \cite{DAntonaMunarini:2000:combinatorial}
    $$
     b_n =
    (-1)^{n-1}
    \begin{vmatrix}
      m_1 & 1 &0&\dotsm&0\\
      m_2 & m_1 & 1&\dotsm&0\\
      \vdots &\vdots &\ddots &\ddots\\
      m_{n-1} & m_{n-2} & \dotsm & m_1 & 1\\
      m_n & m_{n-1} & \dotsm & m_2 & m_1\\
    \end{vmatrix}, 
     \qquad
   m_n =
   \begin{vmatrix}
     b_1 & -1&0&\dotsm&0 \\
      b_2 & b_1 & -1&\dotsm &0\\
      \vdots &\vdots &\ddots &\ddots\\
      b_{n-1} & b_{n-2} & \dotsm & b_1 & -1\\
      b_n & b_{n-1} & \dotsm & b_2 & b_1\\
    \end{vmatrix}.
    $$ 
    \end{enumerate}
\end{rem}

\section{Monotone-to-classical case: }\label{sec8}
Although we have not yet found a satisfactory description for the general coefficients
appearing in the monotone-to-classical cumulant
formula 
\begin{equation}
K_n=\sum_{\pi\in\SP(n)}\beta(\pi) H_{\pi},                 
\end{equation}
we report on some partial results including Theorem \ref{classical-monotone}.

We first provide some general considerations on the recursive nature of this
problem and then apply the approach of \cite{JV13} to obtain some special
cases. 

We first observe that the equation 
\begin{equation}\label{eq start}
\sum_{\pi \in \SP(n)}K_{\pi}(X_1,\dots, X_n)=\varphi(X_1\cdots X_n)=\sum_{\pi \in \NC(n)}\frac{1}{\tau(\pi)!}H_{\pi}(X_1,\dots, X_n)
\end{equation}
can be transofmed into the form
\begin{equation}\label{eq:clamo}
K_n(X_1,\dots,X_n)=-\sum_{\sigma \in \SP(n), \sigma\neq \hat{1}_n}K_{\sigma}(X_1,\dots, X_n)+\sum_{\pi \in \NC(n)}\frac{1}{\tau(\pi)!}H_{\pi}(X_1,\dots, X_n). 
\end{equation}
which allows for a recursive algorithm to obtain the coefficients
$\beta(\pi)$. More precisely, once $\beta$ is known for all partitions of size
$k<n$, the coefficient $\beta(\pi)$ of a partition $\pi\in \SP(n)$ is obtained
as follows: We take any partition $\sigma\in[\pi,\hat{1}_n)$ and express $K_{\sigma}$
in terms of the monotone cumulants. If $\sigma=\{W_1,W_2,\dots,W_s\}$, the
coefficient of $H_{\pi}$ in such an expression will be exactly
$\beta(\pi|_{W_1})\cdots \beta(\pi|_{W_s})$. We must do this for every
$\sigma\geq\pi$ and if $\pi\in \NC(n)$  we must in addition consider the
coefficient $(\tau(\pi)!)^{-1}$ on the right hand side of eq.\ (\ref{eq:clamo})
as well.   
Hence relation \eqref{eq:clamo} above can be recast into the following recursion
\begin{equation}\label{eq:recursion1}
\beta(\pi)
=\begin{cases}
\displaystyle\frac{1}{\tau(\pi)!}-\sum_{\substack{\sigma\in \SP(n) \\ \sigma \in[\pi,\hat{1}_n)}}\prod_{W\in \sigma}\beta(\pi|_W)& \text{if } \pi\in \NC(n),  \vspace{10pt}\\ 
\displaystyle-\sum_{\substack{\sigma\in \SP(n) \\ \sigma \in[\pi,\hat{1}_n)}}\prod_{W\in \sigma}\beta(\pi |_W)& \text{if } \pi\notin \NC(n).
\end{cases}
\end{equation}

We show first that $\beta(\pi)=0$ for all $\pi \notin \irr\SP$ by induction on $\abs{\pi}$. For $\abs{\pi}=1$ the assertion is trivial. Suppose that $\beta(\pi)=0$ for all $\pi \notin \irr\SP$ with $\abs{\pi}<k$.

Now let $\abs{\pi}=k$ with $\pi\notin \irr\SP$. Let $\pi_1,\dots,\pi_s$, $s\geq 2$, be the irreducible components of $\pi$. By formula (\ref{eq:recursion1}), we need to look at partitions $\sigma\in[\pi,1_n)$. If a block $V$ of $\sigma$ contains blocks of $\pi$ from different irreducible components, then $\pi|_V$ is reducible and hence $\beta(\pi|_V)=0$ by induction hypothesis.

Therefore, a contribution to $\beta(\pi)$ can only come from partitions of the
form $\sigma=\sigma_1\cup\sigma_2\cup\dots\cup\sigma_s$, with $\sigma_i\geq
\pi_i$, and hence 

\begin{align}\label{eq:betairr}
\sum_{\substack{\sigma\in \SP(n) \\ \sigma \in[\pi,\hat{1}_n)}}\prod_{W\in \sigma}\beta(\pi|_W)&=\sum_{\substack{\sigma_1\cup\dots\cup\sigma_s\in \SP(n) \\ \sigma_i\geq \pi_i}}\prod_{W\in \sigma}\beta(\pi|_W)\\
&=\prod_{i=1}^s\sum_{\sigma_i\geq \pi_i}\prod_{W\in\sigma_i}\beta(\pi_i|_W). 
\end{align}
We now apply recursion \eqref{eq:recursion1}  separately
to each sum occuring in \eqref{eq:betairr} and obtain
$$
\sum_{\sigma_i\geq \pi_i}\prod_{W\in\sigma_i}\beta(\pi_i|_W)
= \beta(\pi_i) + \sum_{\sigma_i\in [\pi_i, \hat{1})}\prod_{W\in\sigma_i}\beta(\pi_i|_W)
=\begin{cases}
\displaystyle\frac{1}{\tau(\pi_i)!}& \text{if } \pi_i\in \NC,  \vspace{10pt}\\ 
\displaystyle 0& \text{if } \pi_i\notin \NC.
\end{cases}
$$
Now we observe that $\pi\in \NC$ if and only if each $\pi_i \in \NC$
and that $\tau(\pi)!=\tau(\pi_1)!\cdots\tau(\pi_s)!$ for $\pi\in \NC$,
to  conclude that $\beta(\pi)=0$. 

\begin{rem}
Note that, in order to show that $\beta$ is supported on $\irr\SP$, we only used that, for any pair of cumulants $(A_n)_{n\geq 1}$, $(C_n)_{n\geq 1}$, we have:
$$\sum_{\pi\in \SP(n)} \omega_1(\pi) A_{\pi}(\mathbf X)=\varphi(\mathbf X)= \sum_{\pi\in \SP(n)} \omega_2(\pi) C_{\pi}(\mathbf X),$$
where, for $i=1,2$, the weights $\omega_i(\pi)=\omega_i(\pi_1)\cdots \omega(\pi_s)$ factorize according to the irreducible components $\pi_1,\dots,\pi_s$ of $\pi$. Following the proof that $\beta(\pi)=0$ for $\pi\notin \irr\SP$, we get for some constants $(\alpha(\pi))_{\pi \in \irr\SP} \subset \real$ that 
$$
A_n(\mathbf X)=\sum_{\pi\in \irr\SP(n)}\alpha(\pi)C_{\pi}(\mathbf X).
$$
This shows that all 12 cumulant formulas are supported on $\irr\SP$. Moreover, the fact that monotone, Boolean and free cumulants assign a weight $\omega(\pi)=0$ to any crossing partition implies that the corresponding cumulant formulas will be actually supported on $\irr\SP\cap\NC=\irr\NC$.

The classical and free cumulants both have weights that are invariant under cyclic rotations. This means that the corresponding $\alpha(\pi)$ will also be rotationally invariant. Hence $\alpha(\pi)$ can only be nonzero if all cyclic rotations of $\pi$ remain in $\irr\SP(n)$. This means that $\pi\in \conn\SP$.
\end{rem}

Because of the nature of
the recursion \eqref{eq:recursion1}, the dependence of the coefficient
$\beta(\pi)$ is uniquely determined by
crossing/nesting structure of the blocks of partitions contained in the
interval $[\pi,\hat{1}_n]$. Hence we suggest the following refinement of the
anti-interval graph which actually distinguishes between crossings and nestings
(see fig.~\ref{fig:graphs}): 
\begin{defi}\label{defimonotone}
The anti-interval digraph $\overrightarrow{G}(\pi)$ is obtained from the interval graph by replacing every (non-directed) edge $(V_i,V_j)$ of $\tilde{G}(\pi)$ by:
\begin{enumerate}
 \item a directed edge $(V_i,V_j)$, if $V_j$ nests inside $V_i$, 
 \item a directed edge $(V_j,V_i)$, if $V_i$ nests inside $V_j$,
 \item a non-directed edge $(V_i,V_j)$, otherwise (equivalently, if $V_i$ and $V_j$ cross).
\end{enumerate}
\end{defi}
It is not hard to see (by induction on $\abs{\pi}$) that the digraph $\overrightarrow{G}(\pi)=(V,E)$ determines $\beta(\pi)$: In the recursion \eqref{eq:recursion1}, the contribution of each $1_n>\sigma=\{W_1,\dots,W_s\}\geq\pi$, one should look at the subgraphs $\overrightarrow{G}(\pi)|_{W_i}$ of $\overrightarrow{G}(\pi)$ indicated by the blocks of $\sigma$ and then just multiply all $\beta(\pi|_{W_i})$, which are known already since $1_n>\sigma$. 
So we may write $
\beta(\pi)=\beta(\overrightarrow{G}(\pi)). 
$

Let us conclude with the proof of Theorem~\ref{classical-monotone}.
First we will use the approach of \cite{JV13} to obtain a rather explicit
expression for $\beta(\pi)$ from which everything can be deduced.

\begin{lem}
  \label{lem:betaformula}
  \begin{equation}
  \label{eq:betaformula}
  \beta(\pi) =  \sum_{\sigma  \trianglerighteq\pi} \frac{\mu_\SP(\sigma,\hat{1}_n)}{\tau(\pi|\sigma)!}
  \end{equation}
  where $\tau(\pi|\sigma)!=\prod_{W\in\sigma} \tau(\pi|_W)!$
  and the partial order relation $ \pi  \trianglelefteq \sigma$ on $\SP$ 
  refines the
  usual order $\pi\leq\sigma$ by the additional requirement that $\pi|_W$ is
  noncrossing for every block $W\in\sigma$.
\end{lem}
\begin{proof}
  We follow the proof of \cite{JV13} and write
  \begin{align*}
    K_n &= \sum_{\sigma\in\SP(n)} \varphi_\sigma\, \mu_\SP(\sigma,\hat{1}_n) \\
       &= \sum_{\sigma\in\SP(n)} \sum_{\pi \trianglelefteq\sigma}
       \frac{1}{\tau(\pi|\sigma)!} H_\pi\, \mu_\SP(\sigma,\hat{1}_n) \\
       &= \sum_{\pi\in\SP(n)} H_\pi \sum_{\sigma \trianglerighteq\pi}
       \frac{\mu_\SP(\sigma,\hat{1}_n)}{\tau(\pi|\sigma)!}.  
  \end{align*}
\end{proof}

\begin{proof}[Proof of Theorem~\ref{classical-monotone}]
  The first part follows from Weisner's lemma
  \cite[Prop.~6.3]{BarnabeiBriniRota:1986:theory}.
  Let $P$ be a lattice and $a,b,c\in P$, then
  $$
  \sum_{x\wedge a=c}\mu(x,b) 
  =
  \begin{cases}
    0 & \text{if $a\not\geq b$},  \\
    \mu(c,b) &\text{if $a\geq b$}. 
  \end{cases}
  $$
  Consider the function
  $$
  f_\pi(\sigma) =
  \begin{cases}
    \frac{1}{\tau(\pi|\sigma)!} & \text{if  $\pi\trianglelefteq\sigma$}, \\
    0 & \text{if  $\pi\not\trianglelefteq\sigma$}. 
  \end{cases}
  $$
  Assume $\pi\in\SP(n)$ is not irreducible, and let $\rho=\hat\pi\ne\hat{1}_n$ be its interval closure.
  Then it is easy to see that $f_\pi(\sigma)=f_\pi(\sigma\wedge\rho)$.
  Indeed, if the restriction $\pi|_b$ has a crossing for some $b\in\sigma$, 
  then it occurs in the restriction
  $\pi_j|_b$ of some irreducible factor $\pi_j$ of $\pi$.
  If $\pi|\sigma$ has no crossings, then all nesting trees are contained inside
  the irreducible factors and therefore occur inside the blocks of $\sigma\wedge\rho$.
  
  Using this fact we can write
  \begin{align*}
    \beta(\pi)
    &= \sum_{\sigma\trianglerighteq\pi} \frac{\mu_\SP(\sigma,\hat{1}_n)}{\tau(\pi|\sigma)!}\\
    &= \sum_{\sigma\geq\pi} f_\pi(\sigma)\mu_\SP(\sigma,\hat{1}_n)\\
    &= \sum_{\sigma\geq\pi} f_\pi(\sigma\wedge\rho)\mu_\SP(\sigma,\hat{1}_n)\\
    &= \sum_{\tau\geq\pi} f_\pi(\tau) \sum_{\sigma\wedge\rho=\tau}
    \mu_\SP(\sigma,\hat{1}_n)\\
    &= 0.
  \end{align*}
This proves the first statement. 
  
  The second part holds true because the involved trees $\tau(\pi|\sigma)$ 
  are trivial and hence formula \eqref{eq:betaformula} becomes
  $$
  \beta(\pi)= \sum_{\sigma  \trianglerighteq\pi} \mu_\SP(\sigma,\hat{1}_n)=
  (-1)^{\abs{\pi}-1} T_{G(\pi)}(1,0)
  $$
  as in the proof of \cite[Theorem 7.1]{JV13}. Note that the assumptions imply
  that $\pi$ is connected.

  Finally, let $\pi\in\irr\NC$ be an irreducible noncrossing partition of depth $2$.
  This means that there is one outer block and $m:=\abs{\pi}-1$ inner blocks and the nesting
  tree $\tau(\pi)$ consists of the root and $m$ leaves.

  We classify the partitions of $t=\tau(\pi)$ according to the number of
  elements in the block containing the root. There are $\binom{n}{k}$ different
  subsets $b$ of $t$ containing $k$ vertices in addition to the root and for
  every set like this  $(t|_b)!=k+1$. Then the remaining vertices of $t$ can
  be partitioned without affecting the factorial and as $k$ ranges between $0$
  and $n$ we obtain 
  $$
  \beta(t)=\sum_{k=0}^m \binom{m}{k} 
    \sum_{\rho\in\SP(m-k)} 
      \frac{-\abs\rho}{k+1} 
      \mu_\SP(\rho,\hat{1}_{m-k})
  = -\sum_{k=0}^m \binom{m}{k}  \frac{\alpha_{m-k}}{k+1} 
  $$
  where
  \begin{align*}
    \alpha_n &= 
    \sum_{\rho\in\SP(n)} \abs\rho\mu_\SP(\rho,\hat{1}_n)
    = \frac{d}{dx}
    \left.
      \sum_{\rho\in\SP(n)} x^{\abs\rho}\mu_\SP(\rho,\hat{1}_n)
    \right\rvert_{x=1}. 
  \end{align*}
  The derivand in the last expression can be interpreted as the classical
  cumulant of order $n$ of a  
  standard Bernoulli law of weight $x$
  and the exponential generating function therefore is
  \begin{align*}
    \sum_{n=1}^\infty \frac{\alpha_n}{n!} z^n
    &= \frac{d}{dx}
    \left. \log \left( 1 + \sum_{n=1}^\infty \frac{x}{n!} z^n\right)
    \right\rvert_{x=1}
    \\
    &= 
    \left.
      \frac{d}{dx}
      \log (1 + x(e^z-1))
    \right\rvert_{x=1}
    \\
    &=
    \left.
      \frac{e^z-1}{1+x(e^z-1)}
    \right\rvert_{x=1}\\
    &= 1 - e^{-z}  
    = \sum_{n=1}^\infty\frac{(-1)^{n-1}}{n!} z^n
  \end{align*}
  and hence $\alpha_n=(-1)^{n-1}$.
  Thus, denoting by $\beta_m := \beta(t)$, 
  \begin{align*}
    \beta_m
    &= -\sum_{k=0}^m \binom{m}{k}  \frac{(-1)^{m-k-1}}{k+1} 
    \end{align*}
    and so 
    \begin{align*}
    \sum_{m=0}^\infty \frac{\beta_m}{m!} z^m 
    &= \sum_{m=0}^\infty\sum_{k=0}^m \frac{1}{(k+1)!} \frac{(-1)^{m-k}}{(m-k)!} z^m\\
    &= \frac{1}{z}(e^z-1)e^{-z} \\
    &= \sum_{m=0}^\infty \frac{(-1)^m}{(m+1)!}z^m
  \end{align*}
  and consequently $\beta(t)=\beta_m=\frac{(-1)^m}{m+1}= \frac{(-1)^{\abs{\pi}-1}}{\abs{\pi}}$.
\end{proof}

\begin{rem}
For the family $\pi_n=\{\{1,2n\},\{2,2n-1\},\dots,\{n,n+1\}\}$, $n\geq 1$, the
sequence $(n!\beta(\pi_n))_n=(1,-1,4,-33,456,-9460,274800,\dots) $ 
are the coefficients of the log-Bessel function
(\cite[A101981]{Sloane:encyclopedia} ).

Indeed, for $\pi=\pi_n$ the nesting tree is just a line segment and
$$
\beta(\pi) =  \sum_{\sigma  \geq\pi} \frac{\mu_\SP(\sigma,\hat{1}_n)}{[\pi,\sigma]!}.
$$
Here we use the fact that the interval $[\pi,\sigma]$ is isomorphic
to a direct product $\prod \SP(k)^{m_k}$ and
$\tau(\pi|\sigma)!=[\pi,\sigma]! = \prod(k!)^{m_k}$.
In other words $\beta=f*\mu_\SP$, the convolution of the multiplicative functions on
$\SP$ associated to the sequences $f_n=\frac{1}{n!}$ and $\mu_n=(-1)^{n-1}(n-1)!$.

Recall \cite{DoubiletRotaStanley:1972:foundations6,Stanley:1999:enumerative2}
that the reduced incidence 
algebra of $\SP$ incarnates the Faa di Bruno formula for exponential power series, i.e.,
if $A(z) = \sum_{n=1}^\infty \frac{a_n}{n!}\,z^n$ and  $B(z) = \sum_{n=1}^\infty \frac{b_n}{n!}\,z^n$ are
the corresponding exponential generating functions and $c=a*b$,
then the generating function of the sequence $(c_n)$ is
$$
\sum_{n=1}^\infty \frac{c_n}{n!}\,z^n = B(A(z)).
$$
In our case $\beta=f*\mu_\SP$, where
$$
F(z) = \sum_{n=1}^\infty \frac{z^n}{(n!)^2} = J_0(2i\sqrt{z})-1
$$
is the Bessel function of first kind and
$$
M(z) = \sum_{n=1}^\infty (-1)^{n-1}\frac{z^n}{n} = \log(1+z)
$$
is the logarithm, thus $\beta(\pi_n)$ are the coefficients of the log-Bessel function
$$
\log J_0(2i\sqrt{z}).
$$
The sequence $b_n=n!\beta(\pi_n)$ satisfies a recursion found by Carlitz
\cite{Carlitz:1963:sequence},
namely
$$
b_{n+1} = \sum_{k=1}^n \binom{n}{k}\binom{n}{k-1} b_k b_{n+1-k}
$$
and therefore $\beta_n=\beta(\pi_n)$ satisfies
\begin{multline*}
\beta_{n+1} = \frac{b_{n+1}}{(n+1)!} 
=\sum_{k=1}^n \frac{1}{n+1} \frac{n!}{(n-k)!(k-1)!}
\frac{b_k}{k!}\frac{b_{n+1-k}}{(n+1-k)!}
= \sum_{k=1}^n \frac{n}{n+1}\binom{n-1}{k-1} \beta_k\beta_{n+1-k}. 
\end{multline*}
\end{rem}


\begin{thebibliography}{ABGO04}

\bibitem[ABGO04]{AccardiBenGhorbalObata:2004:monotone}
Luigi Accardi, Anis Ben~Ghorbal, and Nobuaki Obata, \emph{Monotone
  independence, comb graphs and {B}ose-{E}instein condensation}, Infin. Dimens.
  Anal. Quantum Probab. Relat. Top. \textbf{7} (2004), no.~3, 419--435.

\bibitem[Bar11]{Barry:2011:eulerian}
Paul Barry, \emph{Eulerian polynomials as moments, via exponential {R}iordan
  arrays}, J. Integer Seq. \textbf{14} (2011), no.~9, Article 11.9.5, 14.

\bibitem[BBLS11]{BelinschiBozejkoLehnerSpeicher:2011:normal}
Serban~T. Belinschi, Marek Bo{\.z}ejko, Franz Lehner, and Roland Speicher,
  \emph{The normal distribution is {$\boxplus$}-infinitely divisible}, Adv.
  Math. \textbf{226} (2011), no.~4, 3677--3698.

\bibitem[BBR86]{BarnabeiBriniRota:1986:theory}
M.~Barnabei, A.~Brini, and Dzh.-K. Rota, \emph{The theory of {M}\"obius
  functions}, Uspekhi Mat. Nauk \textbf{41} (1986), no.~3(249), 113--157,
  Translated from the Italian by S. I. Gelfand. \MR{854241 (87k:05008)}

\bibitem[BN08]{BN08}
Serban~T. Belinschi and Alexandru Nica, \emph{{$\eta$}-series and a {B}oolean
  {B}ercovici-{P}ata bijection for bounded {$k$}-tuples}, Adv. Math.
  \textbf{217} (2008), no.~1, 1--41.

\bibitem[Car63]{Carlitz:1963:sequence}
L.~Carlitz, \emph{A sequence of integers related to the {B}essel functions},
  Proc. Amer. Math. Soc. \textbf{14} (1963), 1--9. \MR{0166147 (29 \#3425)}

\bibitem[CF69]{CartierFoata:1969:problemes}
P.~Cartier and D.~Foata, \emph{Probl\`emes combinatoires de commutation et
  r\'earrangements}, Lecture Notes in Mathematics, No. 85, Springer-Verlag,
  Berlin, 1969, Electronic reedition with three new appendices, Sem. Loth.
  Comb. 2006.

\bibitem[CS86]{CartwrightSoardi:1986:random}
Donald~I. Cartwright and P.~M. Soardi, \emph{Random walks on free products,
  quotients and amalgams}, Nagoya Math. J. \textbf{102} (1986), 163--180.

\bibitem[DM00]{DAntonaMunarini:2000:combinatorial}
Ottavio~M. D'Antona and Emanuele Munarini, \emph{A combinatorial interpretation
  of punctured partitions}, J. Combin. Theory Ser. A \textbf{91} (2000),
  no.~1-2, 264--282, In memory of Gian-Carlo Rota.

\bibitem[DRS72]{DoubiletRotaStanley:1972:foundations6}
Peter Doubilet, Gian-Carlo Rota, and Richard Stanley, \emph{On the foundations
  of combinatorial theory. {VI}. {T}he idea of generating function},
  Proceedings of the {S}ixth {B}erkeley {S}ymposium on {M}athematical
  {S}tatistics and {P}robability ({U}niv. {C}alifornia, {B}erkeley, {C}alif.,
  1970/1971), {V}ol. {II}: {P}robability theory, Univ. California Press,
  Berkeley, Calif., 1972, pp.~267--318. \MR{0403987 (53 \#7796)}

\bibitem[GKP94]{GrahamKnuthPatashnik:1994:concrete}
Ronald~L. Graham, Donald~E. Knuth, and Oren Patashnik, \emph{Concrete
  mathematics}, second ed., Addison-Wesley Publishing Company, Reading, MA,
  1994, A foundation for computer science.

\bibitem[GZ83]{GreeneZaslavsky:1983:interpretation}
Curtis Greene and Thomas Zaslavsky, \emph{On the interpretation of {W}hitney
  numbers through arrangements of hyperplanes, zonotopes, non-{R}adon
  partitions, and orientations of graphs}, Trans. Amer. Math. Soc. \textbf{280}
  (1983), no.~1, 97--126.

\bibitem[HS11a]{HS11b}
Takahiro Hasebe and Hayato Saigo, \emph{Joint cumulants for natural
  independence}, Electron. Commun. Probab. \textbf{16} (2011), 491--506.

\bibitem[HS11b]{HS11a}
\bysame, \emph{The monotone cumulants}, Ann. Inst. Henri Poincar\'e Probab.
  Stat. \textbf{47} (2011), no.~4, 1160--1170.

\bibitem[JV13]{JV13}
Matthieu Josuat-Verg{\`e}s, \emph{Cumulants of the {$q$}-semicircular {L}aw,
  {T}utte {P}olynomials, and {H}eaps}, Canad. J. Math. \textbf{65} (2013),
  no.~4, 863--878.

\bibitem[KP04]{KrishnapurPeres:2004:recurrent}
Manjunath Krishnapur and Yuval Peres, \emph{Recurrent graphs where two
  independent random walks collide finitely often}, Electron. Comm. Probab.
  \textbf{9} (2004), 72--81 (electronic).

\bibitem[Kre72]{Kreweras:1972:partitions}
G.~Kreweras, \emph{Sur les partitions non crois\'ees d'un cycle}, Discrete
  Math. \textbf{1} (1972), no.~4, 333--350.

\bibitem[Leh02]{L02}
Franz Lehner, \emph{Free cumulants and enumeration of connected partitions},
  European J. Combin. \textbf{23} (2002), no.~8, 1025--1031.

\bibitem[Leh04]{L04}
\bysame, \emph{Cumulants in noncommutative probability theory. {I}.
  {N}oncommutative exchangeability systems}, Math. Z. \textbf{248} (2004),
  no.~1, 67--100.

\bibitem[Len10]{L10}
Romuald Lenczewski, \emph{Matricially free random variables}, J. Funct. Anal.
  \textbf{258} (2010), no.~12, 4075--4121.

\bibitem[Len12]{L12}
\bysame, \emph{Matricial {R}-transform}, J. Funct. Anal. \textbf{262} (2012),
  no.~4, 1802--1844.

\bibitem[LM11]{LauveMastnak:2011:primitives}
Aaron Lauve and Mitja Mastnak, \emph{The primitives and antipode in the {H}opf
  algebra of symmetric functions in noncommuting variables}, Adv. in Appl.
  Math. \textbf{47} (2011), no.~3, 536--544.

\bibitem[Mur01]{M01}
Naofumi Muraki, \emph{Monotonic independence, monotonic central limit theorem
  and monotonic law of small numbers}, Infin. Dimens. Anal. Quantum Probab.
  Relat. Top. \textbf{4} (2001), no.~1, 39--58.

\bibitem[Mur02]{Muraki:2002:universal}
\bysame, \emph{The five independences as quasi-universal products}, Infin.
  Dimens. Anal. Quantum Probab. Relat. Top. \textbf{5} (2002), no.~1, 113--134.

\bibitem[NS06]{NS06}
Alexandru Nica and Roland Speicher, \emph{Lectures on the combinatorics of free
  probability}, London Mathematical Society Lecture Note Series, vol. 335,
  Cambridge University Press, Cambridge, 2006.

\bibitem[Oba04]{Obata:2004:quantum}
Nobuaki Obata, \emph{Quantum probabilistic approach to spectral analysis of
  star graphs}, Interdiscip. Inform. Sci. \textbf{10} (2004), no.~1, 41--52.

\bibitem[P{\'o}l21]{Polya:1921:Aufgabe}
Georg P{\'o}lya, \emph{\"{U}ber eine {A}ufgabe der
  {W}ahrscheinlichkeitsrechnung betreffend die {I}rrfahrt im {S}tra\ss ennetz},
  Math. Ann. \textbf{84} (1921), no.~1-2, 149--160.

\bibitem[Pri98]{prisner:1998:journey}
Erich Prisner, \emph{A journey through intersection graph theory}, Lecture
  Notes, 1998, Script of a minicourse at the Universidad de Chile.

\bibitem[Rot64]{Rota:1984:foundations1}
Gian-Carlo Rota, \emph{On the foundations of combinatorial theory. {I}.
  {T}heory of {M}\"obius functions}, Z. Wahrscheinlichkeitstheorie und Verw.
  Gebiete \textbf{2} (1964), 340--368 (1964).

\bibitem[RS00]{RotaShen:2000:combinatorics}
Gian-Carlo Rota and Jianhong Shen, \emph{On the combinatorics of cumulants}, J.
  Combin. Theory Ser. A \textbf{91} (2000), no.~1-2, 283--304, In memory of
  Gian-Carlo Rota.

\bibitem[Sch47]{Schutzenberger:1947:certains}
Marcel-Paul Schutzenberger, \emph{Sur certains param\`etres caract\'eristiques
  des syst\`emes d'\'ev\'enements compatibles et d\'ependants et leur
  application au calcul des cumulants de la r\'ep\'etition}, C. R. Acad. Sci.
  Paris \textbf{225} (1947), 277--278.

\bibitem[Slo14]{Sloane:encyclopedia}
N.~J.~A. Sloane, \emph{The on-line encyclopedia of integer sequences},
  published electronically at \url{http://oeis.org}, 2014.

\bibitem[Spe94]{S94}
Roland Speicher, \emph{Multiplicative functions on the lattice of noncrossing
  partitions and free convolution}, Math. Ann. \textbf{298} (1994), no.~4,
  611--628.

\bibitem[Spe97]{Speicher:1995:universal}
\bysame, \emph{On universal products}, Free probability theory ({W}aterloo,
  {ON}, 1995), Fields Inst. Commun., vol.~12, Amer. Math. Soc., Providence, RI,
  1997, pp.~257--266.

\bibitem[Sta99]{Stanley:1999:enumerative2}
Richard~P. Stanley, \emph{Enumerative combinatorics. {V}ol. 2}, Cambridge
  Studies in Advanced Mathematics, vol.~62, Cambridge University Press,
  Cambridge, 1999, With a foreword by Gian-Carlo Rota and appendix 1 by Sergey
  Fomin.

\bibitem[Sta12]{Stanley:2012:enumerative1}
\bysame, \emph{Enumerative combinatorics. {V}olume 1}, second ed., Cambridge
  Studies in Advanced Mathematics, vol.~49, Cambridge University Press,
  Cambridge, 2012.

\bibitem[SW97]{SW97}
Roland Speicher and Reza Woroudi, \emph{Boolean convolution}, Free probability
  theory ({W}aterloo, {ON}, 1995), Fields Inst. Commun., vol.~12, Amer. Math.
  Soc., Providence, RI, 1997, pp.~267--279.

\bibitem[Voi85]{V85}
Dan Voiculescu, \emph{Symmetries of some reduced free product
  {$C^\ast$}-algebras}, Operator algebras and their connections with topology
  and ergodic theory ({B}u\c steni, 1983), Lecture Notes in Math., vol. 1132,
  Springer, Berlin, 1985, pp.~556--588.

\bibitem[Voi86]{Voiculescu:1986:addition}
\bysame, \emph{Addition of certain noncommuting random variables}, J. Funct.
  Anal. \textbf{66} (1986), no.~3, 323--346.

\bibitem[WH86]{WeissHavlin:1986:comb}
George~H. Weiss and Shlomo Havlin, \emph{Some properties of a random walk on a
  comb structure}, Physica A \textbf{134} (1986), no.~2, 474--482.

\bibitem[Woe86]{Woess:L1986:nearest}
Wolfgang Woess, \emph{Nearest neighbour random walks on free products of
  discrete groups}, Boll. Un. Mat. Ital. B (6) \textbf{5} (1986), no.~3,
  961--982.

\bibitem[Woe00]{Woess:2000:random}
\bysame, \emph{Random walks on infinite graphs and groups}, Cambridge Tracts in
  Mathematics, vol. 138, Cambridge University Press, Cambridge, 2000.

\end{thebibliography}
\providecommand{\bysame}{\leavevmode\hbox to3em{\hrulefill}\thinspace}
\providecommand{\MR}{\relax\ifhmode\unskip\space\fi MR }
\providecommand{\MRhref}[2]{%
  \href{http://www.ams.org/mathscinet-getitem?mr=#1}{#2}
}
\providecommand{\href}[2]{#2}

\end{document}